\documentclass[final, 12pt]{article}

\usepackage{amsmath}
\usepackage{amsfonts}
\usepackage{makeidx}
\usepackage{amsthm}
\usepackage{mathtools}
\usepackage{amssymb}
\usepackage{bm}
\usepackage{cite}
\usepackage{showkeys}
\usepackage[english]{babel}
\usepackage{dblaccnt}
\usepackage{accents}
\usepackage{graphicx}
\usepackage{psfrag}
\usepackage{subfig}
\usepackage{color}
\usepackage{float}
\usepackage{amscd}
\usepackage[mathscr]{euscript}
\usepackage{lipsum}
\usepackage{epstopdf}
\usepackage{algorithm}
\usepackage{algpseudocode}

\newcommand{\R}{\mathbb{R}}
\newcommand{\C}{\mathbb{C}}
\renewcommand{\S}{\mathbb{S}}

\renewcommand{\phi}{\varphi}
\newcommand{\dd}{\mathrm{d}}

\renewcommand{\sc}{\mathrm{sc}}

\renewcommand{\Re}{\mathrm{Re}\,}
\renewcommand{\Im}{\mathrm{Im}\,}

\newtheorem{defi}{Definition}

\newtheorem{theorem}[defi]{Theorem}

\newtheorem{assumption}[defi]{Assumption}
\newtheorem{remark}[defi]{Remark}
\begin{document}

\title{On the numerical solution to   an   inverse medium     problem}
\author{Dinh-Liem Nguyen\thanks{Department of Mathematics, Kansas State University, Manhattan, KS 66506, USA; (\texttt{dlnguyen@ksu.edu}, \texttt{trungt@ksu.edu})}
\and Trung Truong\footnotemark[1]
 }

\date{}
\maketitle

\begin{center}
\vspace{-0.6cm}
\textit{Dedicated to Professor Duong Minh Duc on the occasion of his 70th birthday.}
\end{center}

\begin{abstract}
This paper is concerned with  the inverse medium problem of 
determining the location and  shape of penetrable scattering objects from measurements of
the scattered field. We study a sampling indicator function for recovering the scattering object in a fast and robust way.
  A  flexibility of this indicator function  is  that it is applicable to    data measured in  near-field regime or far-field regime. 
The implementation of the function is simple and does not involve solving any ill-posed problems.
 The resolution analysis and stability estimate   of the indicator function are investigated using   the 
factorization analysis  of the far-field operator along with the Funk-Hecke formula.   
The performance of the method is verified on both simulated and experimental data. 
\end{abstract}

\sloppy

{\bf Keywords. }
 sampling indicator function, inverse medium scattering, near-field data, Cauchy data, sampling method

\bigskip

{\bf AMS subject classification. }
 35R30,  35R09, 65R20

\section{Introduction}
We consider the inverse medium scattering problem for the Helmholtz equation in $\R^n$ ($n= 2$ or 3). 
This inverse problem can be considered as a model problem for the inverse scattering of time-harmonic acoustic waves or time-harmonic TE-polarized electromagnetic waves from bounded inhomogeneous media. It has been one of the central problems in inverse scattering theory and has a wide range of applications including nondestructive testing, radar imaging, medical imaging, and geophysical exploration~\cite{Colto2013}. Needless to say, there has been a large body of literature  on both theoretical and numerical studies on this inverse problem, see~\cite{Colto2000c, Colto2013} and references therein. 

In the present paper, we are  
interested in determining the location and shape of scattering objects from (near-field or far-field) multi-static data of the scattered field.  Since
we study sampling methods  to numerically solve this inverse problem, we will mainly discuss  related results in this direction. The Linear Sampling Method (LSM) can be considered as the first sampling method developed to solve the inverse problem under consideration~\cite{Colto1996}. The LSM aims to  construct an indicator  function for  unknown scattering objects. This indicator function is  evaluated on    sampling points obtained by discretizing   some domain in which the unknown target is searched for. 
The evaluation of the indicator function is typically fast, non-iterative and its construction  does not require advanced a priori information about the unknown target. These are also the  main advantages of the LSM over nonlinear optimization-based methods in solving inverse scattering problems. Shortly after the finding of the LSM, other sampling type  methods 
for  inverse  problems including the point source method~\cite{Potth1996}, the Factorization method (FM)~\cite{Kirsc1998}, the probe method~\cite{Ikeha1998b}  were also developed. We refer to~\cite{Potth2006} for a  discussion on sampling and probe methods studied until 2006. These  methods have been later extended to solve various inverse problems, see~\cite{Potth2006, Kirsc2008,   Cakon2011} and references therein. 

Our work in this paper is inspired by a class of sampling methods that have been studied more recently. We are particularly 
interested in the orthogonality sampling method (OSM)  proposed in~\cite{Potth2010}. While inheriting the advantages of the classical sampling methods mentioned above, the OSM is particularly attractive thanks to its simplicity and efficiency. For instance, the implementation of the OSM  only involves an evaluation of an inner product or some double integral (no need to solve an ill-posed problems). The method  is  extremely robust with respect to noise in the data and its stability can be easily justified. However, the theoretical analysis of the OSM is far less developed compared with that of the classical sampling methods, especially the FM and LSM.  We also refer to~\cite{Liu2017, Harri2022, Ito2012, Ito2013, Harri2019, Kang2018, Park2018} for studies  on direct sampling methods (DSM) which are closely related to the OSM. 
 
Most of the published results on the OSM and DSM deal with the case of far-field data, see, e.g., ~\cite{Potth2010,Gries2011, Ito2012, Liu2017, Harri2019} for results on  the scalar Helmholtz equation  and~\cite{Ito2013, Nguye2019, Harri2020, Le2022} for results on the Maxwell's equations. 
There have been only a few results on the OSM and DSM concerning the case of near-field data. The near-field OSM studied in~\cite{Akinc2016} is only  applicable to the 2D case with circular measurement boundaries. The 3D  case was studied in~\cite{Kang2021} under the small volume hypothesis of well-separated inhomogeneities. A flexibility of the  sampling indicator function   studied in this paper is that it works for near-field data  (and also far-field data) and is not limited to small scatterers or 2D circular measurement boundaries. However, the method requires  Cauchy data instead of  only scattered field data in the  near-field regime.

% 
%We propose in this paper  two novel imaging functionals for solving the inverse problem of interest. 
%These imaging functionals are inspired by the OSM. While  having all the advantages of the OSM originally studied in~\cite{Potth2010}, these imaging functionals can image extended scattering objects, and are applicable to both near-field and far-field  data. More precisely, our imaging functionals use (near-field or far-field) boundary   Cauchy data. Moreover, they can also be easily modified to use only the scattered field data in the case of far-field measurements. 

We analyze the sampling indicator function using the factorization analysis of the far-field operator   and the Funk-Hecke formula. The idea is to relate the indicator function to $\|F \phi_z\|^2$ where $F$ is the far-field operator and $\phi_z$ is some special test function. Then the resolution analysis is investigated  using a factorization of $F$,  analytical properties of the operators in the factorization and the Funk-Hecke formula. To our knowledge, 
the idea of combining the factorization analysis and the Funk-Hecke formula to analyze   sampling indicator functions was initially introduced in~\cite{Vansk2008}. 

%It can be seen from our numerical study  that the  new imaging functionals work very efficiently for both near field and far field data, and are extremely robust with respect to noise in the data. Also, although the formulae of the imaging functionals are not similar and their justifications are completely different, the reconstruction results obtained from the two imaging functionals are quite similar. However, it is still an open question  whether there is an equivalent relation between these two functionals. Using the  analysis of the Factorization method, the justification of the first imaging functional is more satisfactory compared with that of the second one since the effect of using multiple incident waves is not really reflected in the justification of the second imaging functional. However, we have to assume that  the wave number is not an interior transmission eigenvalue in order to apply the analysis of the Factorization method (see in addition Assumption~\ref{assume1}).  

The paper is organized as follows. We will formulate the inverse medium problem of interest and the factorization analysis  in Section~\ref{sec2}. The analysis and stability  of the  sampling indicator function  is presented in Section~\ref{sec3}. Section~\ref{sec5} is dedicated to a numerical study of the sampling method.

\label{intro}

\section{The inverse medium problem and  the factorization analysis }
\label{sec2}
In this section we formulate the inverse problem  of interest  and review some necessary ingredients of the factorization analysis.
This factorization analysis was initially studied for the classical factorization method by Kirsch~\cite{Kirsc1998}. We  refer to~\cite{Kirsc2008} 
for more results about the factorization method. 
Consider a penetrable inhomogeneous medium that 
occupies a bounded Lipschitz domain $D \subset \mathbb{R}^{n}$ ($n = 2$ or 3).  Assume that 
this medium is 
characterized by the bounded function $\eta(y)$ and that $\eta = 0$
in $\R^n \setminus \overline{D}$. Consider the incident plane wave
$$
u_{\mathrm{in}}(x,d) = e^{ikx\cdot d}, \quad  x \in \R^n, \quad d \in  \mathbb{S}^{n-1}: = \{x \in \R^n: |x| = 1\},
$$
where $k>0$ is the wave number and $d$ is the direction vector of propagation.    We consider the following model problem for the scattering of  $u_{\mathrm{in}}(x,d)$ by the  inhomogeneous medium
\begin{align}
\label{Helm}
& \Delta u + k^{2}(1 + \eta(x)) u =0,\quad  x \in \mathbb{R%
}^{n},  \\
& u = u_\sc + u_\mathrm{in}, \\
 \label{radiation}
& \lim_{r\rightarrow \infty }r^{\frac{n-1}{2}}\left( \frac{\partial u_{\mathrm{sc}}}{
\partial r}-iku_{\mathrm{sc}}\right) =0,\quad r=|x|,
\end{align}
where $u(x,d)$ is the total field, $u_\sc(x,d)$ is the scattered field, and the Sommerfeld radiation condition~\eqref{radiation} holds
uniformly for all  directions $x/|x| \in \S^{n-1}$. 
If $\mathbb{R}^n \setminus \overline{D}$ is connected and $\Im(\eta) \geq 0$,  this scattering problem is known to have a unique weak solution $u_\sc \in H^1_{\mathrm{loc}}(\mathbb{R}^n)$, see~\cite{Colto2013}. 

\textbf{Inverse problem.} Consider a Lipschitz domain $\Omega \subset \R^n$ such that $D \subset \Omega$
and denote by $\nu(x)$ the outward normal unit vector to $\partial \Omega$ at $x$. 
We aim to determine $D$ from  $u_\sc(x,d)$  and $ \partial u_\sc(x,d)/ \nu(x)$ for almost all $(x,d) \in \partial \Omega \times \S^{n-1}$.

We denote by $\Phi(x,y)$ the free-space Green's function of the scattering problem~\eqref{Helm}--\eqref{radiation}. It is well known that 
\begin{equation} 
\label{green}
 \Phi(x,y)= 
\begin{cases}
\frac{i}{4}H^{(1)}_0(k|x-y|), & \text{in } \R^2 , \\ 
\frac{e^{ik|x-y|}}{4\pi|x-y|}, & \text{in } \R^3.
\end{cases}
\end{equation}
 It is also well known that  problem~\eqref{Helm}--\eqref{radiation} is equivalent to the Lippmann-Schwinger  integral equation
\begin{align}
\label{LS}
u_\sc(x) = k^2 \int_D \Phi(x,y) \eta(y) u(y)\dd y,
\end{align}
and that the scattered field
has the asymptotic behavior
$$
u_\sc(x,d) = \frac{e^{ik|x|}}{|x|^{(n-1)/2}}\left(u^\infty(\hat{x},d) + O\left(\frac{1}{|x|}\right)\right), \quad |x| \to \infty,
$$
for all $(\hat{x},d) \in \S^{n-1}\times \S^{n-1}$. The function $u^\infty(\hat{x},d)$ is called the scattering amplitude or the far-field pattern of the scattered field $u_\sc$. 
Let $F:L^2(\S^{n-1}) \to L^2(\S^{n-1})$ be the far-field operator defined by 
$$
Fg(\hat{x}) = \int_{\S^{n-1}} u^\infty(\hat{x},d) g(d) \dd s(d).
$$
Thanks to the well-posedness of the scattering problem~\eqref{Helm}--\eqref{radiation} we can define the solution operator $G:L^2(D) \to L^2(\S^{n-1})$ as
\begin{align}
\label{solution}
Gf = w^\infty,
\end{align}
where $w^\infty$ is the scattering amplitude of the unique solution $w$ to
\begin{align}
\label{helm2}
& \Delta w + k^{2}(1 + \eta) w =-k^2\eta f,\quad  \text{in } \mathbb{R}^{n},  \\
\label{rad}
& \lim_{r\rightarrow \infty }r^{\frac{n-1}{2}}\left( \frac{\partial w}{
\partial r}-ikw \right) =0,\quad r=|x|.
\end{align}
Note that this problem is just problem~\eqref{Helm}--\eqref{radiation} rewritten for the scattered field 
with incident field $u_\mathrm{in}$ replaced by $f$. By linearity of    problem~\eqref{Helm}--\eqref{radiation}, 
$Fg$ is just the scattering amplitude of solution $w$ to problem~\eqref{helm2}--\eqref{rad} with $f = v_g$, defined by
$$
v_g(x) =   \int_{\S^{n-1}} e^{ikx \cdot d} g(d) \dd s(d), \quad g \in L^2(\S^{n-1}),  \quad x\in \R^n.
$$
Now we define the compact operator $H: L^2(\S^{n-1}) \to L^2(D)$ as $
Hg = v_g|_D.$
Then obviously the far-field operator can be factorized as
$$
F = GH.
$$
Let $H^*: L^2(\S^{n-1}) \to L^2(\S^{n-1})$ be the adjoint of $H$ given by
$$
H^*g(\hat{x}) =  \int_{\S^{n-1}} e^{-ik\hat{x} \cdot y} g(y) \dd s(y),
$$
and we define $T: L^2(D) \to L^2(D)$ as
\begin{align}
\label{T}
Tf = k^2\eta(f + w),
\end{align}
where $w$ solves problem~\eqref{helm2}--\eqref{rad}. Since $w$ solves the Lippmann-Schwinger equation
$
w(x) = k^2 \int_D  \Phi(x,y) \eta(y)(w(y) + f(y))\dd y,
$
we can deduce from scattering amplitude of $w$  that (see~\cite{Cakon2016})
$$
G = H^*T, \quad F = H^*TH.
$$
To proceed further with the analysis of the Factorization method we need to briefly discuss the  interior  transmission eigenvalues.
We call $k>0$ an interior transmission eigenvalue if the  problem 
\begin{align*}
&\Delta u + k^2 (1+\eta)u = 0, \quad \text{in } D \\
&  \Delta v + k^2 v = 0, \quad \text{in } D \\
&u = v, \quad \frac{\partial u}{\partial \nu} = \frac{\partial v}{\partial \nu}, \quad \text{on } \partial D
\end{align*}
has a nontrivial  solution 
$(u, v) \in L^2(D) \times L^2(D)$ such that $u - v \in H^2(D)$.

We refer to \cite{Cakon2016} and the references therein for more details about  transmission eigenvalues. For the next results, we assume  that $k$ is not an interior transmission eigenvalue.  
The following assumption is also important for the factorization  analysis.
\begin{assumption}
\label{assume1}
We assume that $\eta \in L^\infty(\R^n)$, $\Im(\eta) \geq 0$ and that there
exists a constant $c > 0$ such that $\Re(\eta(x)) + \Im(\eta(x))\geq c$ for almost all $x \in D$.
\end{assumption}
The following theorem of the factorization analysis is important to  the sampling method studied in the next section, see~\cite{Audib2014} for a proof of the theorem.
\begin{theorem}\label{fm-results}
 If Assumption~\ref{assume1} holds true, then operator $T$
defined in~\eqref{T} satisfies the coercivity property. That means there exists a constant $\gamma>0$
such that
$$ 
|\langle Tf, f \rangle| \geq \gamma \|f\|^2, \quad \text{for all } f\in \mathrm{Range}(H).
$$
\end{theorem}

\section{A sampling indicator  function  }
\label{sec3}
In this section we introduce the sampling indicator  function and analyze its properties. 
We   define the indicator  function  $I(z)$ as
\begin{equation}
\label{I1}
I(z) := \int_{\S^{n-1}}\left| \int_{\S^{n-1}} \int_{\partial \Omega} u_\sc(y,d)  \frac{\partial \Phi^\infty(\hat{x},y)}{\partial \nu(y)} - \frac{\partial u_\sc(y,d)}{\partial \nu(y)}   \Phi^\infty(\hat{x},y)\dd s(y) \phi_z(d)\dd s(d) \right|^2 \dd s(\hat{x})
\end{equation}
where $\varphi_z$ is given by
\begin{align}
\label{phi}
\phi_z(d) = e^{-ikd\cdot z}, \quad d \in \S^{n-1}, \quad z \in \R^n,
\end{align}
and $\Phi^\infty(\hat{x},y)$ is the scattering amplitude of the Green's function $\Phi(x,y)$, given by
\begin{equation*} 
\label{}
 \Phi^\infty(\hat{x},y)= 
\begin{cases}
\frac{e^{i\pi/4}}{\sqrt{8\pi k}}e^{-ik\hat{x}\cdot y}, & \text{in } \R^2 , \\ 
\frac{1}{4\pi}e^{-ik\hat{x}\cdot y}, & \text{in } \R^3.
\end{cases}
\end{equation*}
Recall that $J_0$ and $j_0$ are respectively a Bessel function and a spherical Bessel function of the first kind.
The behavior of  $I(z)$ is analyzed in the following theorem.
\begin{theorem}
Assume that $k$ is not an interior transmission eigenvalue and that Assumption~\ref{assume1} holds true. Then
the indicator  function $I(z)$  satisfies
\begin{align}
\label{bound}
&0< \frac{\gamma^2}{ |{\S^{n-1}}|} \left(\int_D |\alpha_z(x)|^2 \dd x\right)^2 \leq I(z) \leq \|G\|^2  \int_D |\alpha_z(x)|^2 \dd x, \quad z\in \R^n,
\end{align}
where $\gamma$ is the positive constant in the coercivity of operator $T$ in Theorem~\ref{fm-results}, $G$ is the solution operator defined in~\eqref{solution}, and
\begin{equation*} 
\alpha_z(x) =
\begin{cases}
2\pi J_0(k|z-x|), & \text{in } \R^2 , \\ 
4\pi j_0(k|z-x|), & \text{in } \R^3.
\end{cases}
\end{equation*}
Furthermore 
\begin{align}
\label{decayrate}
I(z) = O\left(\frac{1}{\mathrm{dist}(z,D)^{n-1}} \right)\quad  \text{ as } \mathrm{dist}(z,D) \to \infty,
\end{align}
where $\mathrm{dist}(z,D)$ is the distance from $z$ to $D$. 
\end{theorem}
\begin{remark}
From the behavior of the Bessel functions $J_0$ and $j_0$ we know that  $|\alpha_z(x)|^2$ peaks as sampling point $z$ approaches point $x$ in the scatterer $D$ and that  $|\alpha_z(x)|^2$
decays as $z$ is away from $x$ with the decay rate~\eqref{decayrate}.  We thus expect from the upper bound in~\eqref{bound}  that $I(z)$ takes small values as $z$ is outside $D$.  
From the lower bound in~\eqref{bound},  $I(z)$ is bounded by a positive constant as $z$ is inside $D$. This is  not a rigorous justification for the behavior of $I(z)$. 
Such a justification  is still an open problem. 

\end{remark}
\begin{proof}
From the Helmholtz integral representation for $u_\sc$ (see~\cite{Colto2013}) we have
$$
u_\sc(x,d) = \int_{\partial \Omega} u_\sc(y,d)   \frac{\partial \Phi(x,y)}{\partial \nu(y)} -\frac{\partial u_\sc(y,d)}{\partial \nu(y)}  \Phi(x,y)\dd s(y).
$$
This deduces  that 
$$
u^\infty(\hat{x},d) =\int_{\partial \Omega} u_\sc(y,d)  \frac{\partial \Phi^\infty(\hat{x},y)}{\partial \nu(y)} - \frac{\partial u_\sc(y,d)}{\partial \nu(y)}   \Phi^\infty(\hat{x},y)\dd s(y).
$$
Then substituting this formula of $u^\infty$ in the far-field operator $F$ implies that
$$
F\phi_z(\hat{x}) =  \int_{\S^{n-1}} \int_{\partial \Omega} u_\sc(y,d)  \frac{\partial \Phi^\infty(\hat{x},y)}{\partial \nu(y)} - \frac{\partial u_\sc(y,d)}{\partial \nu(y)}   \Phi^\infty(\hat{x},y) \dd s(y) \phi_z(d) \dd s(d).
$$
Therefore we derive from the definition of $I(z)$ that
$$
I(z) = \int_{\S^{n-1}} |F\phi_z(\hat{x}) |^2 ds(\hat{x}) = \|F\phi_z\|^2.
$$
Since $\|\phi_z\|^2 = \int_{\S^{n-1}}|e^{-ikz\cdot d}|^2 \dd s(d) = |{\S^{n-1}}|$ (the surface area of ${\S^{n-1}}$), using the Cauchy-Schwarz inequality 
and the factorization of the far-field operator $F$ we obtain
$$
\sqrt{ |{\S^{n-1}}|} \|F\phi_z\| \geq \langle F\phi_z, \phi_z \rangle  = \langle H^*TH\phi_z, \phi_z \rangle = \langle TH\phi_z, H\phi_z\rangle.
$$
Using the coercivity of $T$ in Theorem~\ref{fm-results} and  $\|F\phi_z\| \leq \|G\|\|H\phi_z\|$ implies that
$$
\frac{\gamma^2}{ |{\S^{n-1}}|} \|H\phi_z\|^4 \leq I(z) \leq \|G\|^2  \|H\phi_z\|^2,
$$
where $\gamma$ is the constant from the coercivity  of $T$ in Theorem~\ref{fm-results}. 
% Now let $z \in \Omega$. Then by  Theorem \ref{fm-results} again we have $\phi_{z} \in \mathrm{Range}(H^*)$ which implies that
% $\phi_{z} = H^*h_{z}$ for some $h_{z} \neq 0$. We estimate
%\begin{align*}
% \big\| H \phi_{z} \big\| 
% \geq  \frac{\langle  H \phi_{z},h_{z} \rangle}{\| h_{z}\|}  
% & =  \frac{\langle   \phi_{z},H^*h_{z} \rangle}{\| h_{z}\|}   =  \frac{|{\S^{n-1}}|}{\| h_{z}\|} >0.
%\end{align*}
%This deduces that  $I(z) > \gamma_z$ for  $z \in D$ with $\gamma_z = c^2 |{\S^{n-1}}|^3 / \| h_{z}\|^4$. 

Now  using the Funk-Hecke formula (see~\cite{Colto2013}) we obtain
\begin{equation} 
H\phi_z(x) =  \int_{\S^{n-1}} e^{-ik(z-x) \cdot d}  \dd s(d) =
\begin{cases}
2\pi J_0(k|z-x|), & \text{in } \R^2 , \\ 
4\pi j_0(k|z-x|), & \text{in } \R^3,
\end{cases}
\end{equation}
which allows us to establish the estimate in~\eqref{bound}. The strict positivity of the lower bound in 
the estimate can be deduced from the fact that the operator $H$ is an injective operator, see~\cite{Audib2014}. 
%Now for $z \notin D$, using the Funk-Hecke formula (see~\cite{Colto2013}) we obtain
%$$
%H\phi_z(x) =  \int_{\S^{n-1}} e^{-ik(z-x) \cdot d}  \dd s(d) = \lambda j_0(k|z-x|),
%$$
%Now, recall the 
%$$ \int_{\S^{n-1}} Y^{m}_{\ell} (d) e^{-ikd \cdot y} \, \dd s(d) =\lambda_\ell Y^{m}_{\ell} (\hat{y}) j_\ell (k|y|) \quad \text{ for } \quad \ell \in \N, \,\, m = -\ell, \dots , \ell$$
Finally, using the asymptotic behavior of $J_0(r) = O(1/\sqrt{r})$ and $j_0(r) = O(1/r)$ as $r \to \infty$ we obtain that
$$
\|H\phi_z\|^2 
 = O\left(\frac{1}{\mathrm{dist}(z,D)^{n-1}} \right), \quad \text{as }\mathrm{dist}(z,D) \to \infty.
$$
This completes the proof.
\end{proof}

In   practice  the data are always perturbed with some noise. We assume the noisy data $u^\delta_{\sc}$ and $\partial u^\delta_{\sc}/\partial \nu$ satisfy    
\begin{align}
\label{noise1}
\| u_\sc - u^\delta_{\sc} \|_{L^2(\partial \Omega\times \S^{n-1})}  \leq \delta_1 \| u_\sc \|_{L^2(\partial \Omega\times \S^{n-1})}, \\
\label{noise2}
 \left\| \frac{\partial u_\sc}{\partial \nu} - \frac{\partial u^\delta_{\sc}}{\partial \nu} \right \|_{L^2(\partial \Omega\times \S^{n-1})}  \leq \delta_2 \left\|  \frac{\partial u_{\sc}}{\partial \nu} \right \|_{L^2(\partial \Omega\times \S^{n-1})} ,
\end{align}
for some positive constants $\delta_1, \delta_2 $. We now prove a stability estimate for the indicator  function  $I(z)$.

\begin{theorem}
\label{stability} 
%For all $\widehat{\x},\di \in \mathbb{S}^2$,  let $\mathcal{D}(\widehat{\x},\di) := \u^\infty(\widehat{\x},\di)(\di\times\p)\times \di$ be the far-field data 
%in the imaging functional $I_{OSM}$ in~\eqref{indicator_OSM}. 
Denote   by $I^{\delta}(z)$ the indicator function corresponding to noisy data $ u^\delta_{\sc}$ and $\partial u^\delta_{\sc}/\partial \nu$. Then
\begin{align*}
|I(z) - I^{\delta}(z)| \leq  \mathcal{C} \left( 2\max(\delta_1,\delta_2) + \max(\delta_1^2,\delta_2^2) \right), 
\quad \text{for all } z \in \R^3,
\end{align*}
where 
$$
\mathcal{C} =  |\S^{n-1}|^2 \left(  \| \Phi^\infty \|^2_{L^2( \S^{n-1} \times \partial  \Omega)} +  \left\|  \frac{\partial  \Phi^\infty }{\partial \nu} \right \|^2_{L^2( \S^{n-1} \times\partial  \Omega)} \right) \left(  \| u_\sc \|^2_{L^2(\partial \Omega\times \S^{n-1})} +  \left\|  \frac{\partial u_{\sc}}{\partial \nu} \right \|^2_{L^2(\partial \Omega\times \S^{n-1})} \right).
$$
\end{theorem}
\begin{proof} 
Let $u^{\infty,\delta}(\hat{x},d)$ and $F^\delta$ be the  scattering amplitude and the far-field operator for  noisy Cauchy data. That means
\begin{align}
u^{\infty,\delta}(\hat{x},d) &= \int_{\partial \Omega} u^\delta_\sc(y,d)  \frac{\partial \Phi^\infty(\hat{x},y)}{\partial \nu(y)} - \frac{\partial u^\delta_\sc(y,d)}{\partial \nu(y)}   \Phi^\infty(\hat{x},y)\dd s(y) \\
F^\delta \phi_z(\hat{x}) &=  \int_{\S^{n-1}} u^{\infty,\delta}(\hat{x},d) \phi_z(d) \dd s(d).
\end{align}
Using the Cauchy-Schwarz inequality we have
\begin{align*}
|u^\infty(\hat{x},d) - u^{\infty,\delta}(\hat{x},d)| \leq \|(u_\sc - u_\sc^\delta)(\cdot,d)\| \left\|  \frac{\partial \Phi^\infty(\hat{x},\cdot)}{\partial \nu(\cdot)}  \right\|  +  \left\|  \frac{\partial (u_\sc - u_\sc^\delta)(\cdot,d)}{\partial \nu(\cdot)}  \right\|   \|\Phi^\infty(\hat{x},\cdot)\|,
\end{align*}
and hence
\begin{align*}
&|u^\infty(\hat{x},d) - u^{\infty,\delta}(\hat{x},d)|^2 \\
& \leq \left(\|(u_\sc - u_\sc^\delta)(\cdot,d)\|^2    +  \left\|  \frac{\partial (u_\sc - u_\sc^\delta)(\cdot,d)}{\partial \nu(\cdot)}  \right\|^2 \right)   \left ( \|\Phi^\infty(\hat{x},\cdot)\|^2 + \left\|\frac{\partial \Phi^\infty(\hat{x},\cdot)}{\partial \nu(\cdot)}  \right\|^2 \right).
\end{align*}
Let $C =   \| \Phi^\infty \|^2_{L^2( \S^{n-1} \times \partial \Omega)} +  \left\|  \partial  \Phi^\infty/\partial \nu \right \|^2_{L^2( \S^{n-1} \times \partial \Omega)} $. This leads to 
\begin{align*}
\|u^\infty - u^{\infty,\delta} \|^2_{L^2(\S^{n-1}\times\S^{n-1})} 
& \leq  
C\left(\|(u_\sc - u_\sc^\delta)\|^2_{L^2(\partial \Omega\times \S^{n-1})}    +  \left\|  \frac{\partial (u_\sc - u_\sc^\delta)}{\partial \nu}  \right\|^2_{L^2(\partial \Omega\times \S^{n-1})}  \right)   \\
& \leq C \left( \delta_1^2 \| u_\sc \|^2_{L^2(\partial \Omega\times \S^{n-1})} + \delta_2^2 \left\|  \frac{\partial u_{\sc}}{\partial \nu} \right \|^2_{L^2(\partial \Omega\times \S^{n-1})} \right),
\end{align*}
which implies that
\begin{align*}
\|F\phi_z - F^\delta \phi_z\|^2  \leq C  | \S^{n-1}|^2 \max(\delta_1^2,\delta_2^2) \left(  \| u_\sc \|^2_{L^2(\partial \Omega\times \S^{n-1})} +  \left\|  \frac{\partial u_{\sc}}{\partial \nu} \right \|^2_{L^2(\partial \Omega\times \S^{n-1})} \right).
\end{align*}
Similarly we also have
$$
\|F\phi_z\|^2  \leq C  | \S^{n-1}|^2 \left(  \| u_\sc \|^2_{L^2(\partial \Omega\times \S^{n-1})} +  \left\|  \frac{\partial u_{\sc}}{\partial \nu} \right \|^2_{L^2(\partial \Omega\times \S^{n-1})} \right).
$$
%The proof is similar to that  Lemma 7 in~\cite{Harri2020}. It is presented here for the convenience of the readers. 
Using $I(z) 
= \|F\phi_z\|^2$ and the triangle inequality we have
%\begin{align*}
%I_{}(z) - I_{\delta}(z) &=  \|F\phi_{z}\|^2- \| F_\delta\phi_{z}\|^2 
%\leq \| F\phi_{z} -F_\delta\phi_{z} \|\left( \| F\phi_{z} \|+ \| F_\delta\phi_{z} \| \right)\\
%&\leq |\S^{n-1}| \| F -F_\delta \| \left( 2\| F\phi_{z} \| +  \|F\phi_{z} -F_\delta\phi_{z} \| \right)\\
%&\leq|\S^{n-1}|  \| F\| \delta \left( 2 |\S^{n-1}| \|F\|  +|\S^{n-1}|  \| F\| \delta \right)\leq |\S^{n-1}|^2 \|F \|^2 (\delta^2 + 2\delta),
%\end{align*}
\begin{align*}
|I_{}(z) - I^{\delta}(z)| &= | \|F\phi_{z}\|^2- \| F^\delta\phi_{z}\|^2 |
\leq \| F\phi_{z} -F^\delta\phi_{z} \|\left( \| F\phi_{z} \|+ \| F^\delta\phi_{z} \| \right)\\
&\leq 2\| F\phi_{z} \| \| F\phi_{z} -F^\delta\phi_{z} \| +  \|F\phi_{z} -F^\delta\phi_{z} \|^2\\
&\leq  C |\S^{n-1}|^2 \left(  \| u_\sc \|^2_{L^2(\partial \Omega\times \S^{n-1})} +  \left\|  \frac{\partial u_{\sc}}{\partial \nu} \right \|^2_{L^2(\partial \Omega\times \S^{n-1})} \right)  \left( 2\sqrt{\max(\delta_1^2,\delta_2^2)} + \max(\delta_1^2,\delta_2^2) \right).
\end{align*}
proving the theorem.
\end{proof}

\begin{remark}
We note that if the far-field measurements are taken on the boundary of the ball of large radius $R$, 
by the radiation condition we can approximate $\partial u_\sc/\partial \nu$ by $ik u_\sc$ 
in  $I(z)$. Then the modified indicator function 
\begin{equation}
\label{Ifar}
I_\mathrm{far}(z) := \int_{\S^{n-1}}\left| \int_{\S^{n-1}} \int_{\partial \Omega} u_\sc(y,d)  \frac{\partial \Phi^\infty(\hat{x},y)}{\partial \nu(y)} -ik  u_\sc(y,d)   \Phi^\infty(\hat{x},y)\dd s(y) \phi_z(d)\dd s(d) \right|^2 \dd s(\hat{x})
\end{equation}
only needs  
the scattered field data $u_\sc(x,d)$ and approximates the indicator function $I(z)$.
\end{remark}

\section{Numerical study}
\label{sec5}
In this section we study the numerical performance of the sampling method for both simulated and experimental data in two dimensions. More precisely, for simulated data, we will examine the performance of the method  for data with different wave numbers (Figure~\ref{fi1}), highly noisy data (Figure~\ref{fi2}), far-field data (Figure~\ref{fi3}), and  limited aperture  data (Figure~\ref{fi4}). 
Reconstruction results using  the  indicator function $I_{\mathrm{far}}(z)$  are also presented in the case of  far-field data. For experimental data, we apply the indicator function  $I_{\mathrm{far}}(z)$  to  three data sets of dielectric and metallic objects from the Fresnel Institute (Figure~\ref{fi5}). 
 For the pictures in this section, the indicator functions are scaled by dividing by their maximal values. 
  
  The following common parameters and notations are used in the numerical examples of simulated data
\begin{align*}
&\text{Sampling  domain} = (-2,2)\times(-2,2), \\
&\text{Number of sampling points} = 96^2,\\
 &\partial \Omega = \{ (x_1,x_2)^\top \in \R^2: x_1^2 + x_2^2 = R^2 \}, \\
& \text{Near-field data: } R = 3, \\
 & \text{Far-field data: } R = 100, \\
 & \text{Number of data points on } \partial \Omega\text{: } N_x,\\
& \text{Number of incident plane waves: } N_d.
\end{align*}
  The following scattering objects are considered in the numerical examples.
  
\noindent
a) Kite-shaped  object 
\begin{align*}
\partial D &= \{x \in  \R^2: x=((\cos(t) + 0.65\cos(2t) - 0.65)/2,  1.5\sin(t)/2.5)^\top,  0\leq t \leq 2\pi\}, \\
\eta(x)  &= 0.5 + 0.1i \quad \text{in } D.
\end{align*}
b) Disk-and-rectangle   object 
\begin{align*}
\mathrm{disk} &= \{(x_1,x_2)^\top \in \R^2: (x_1+0.6)^2 + (x_2-0.6)^2 < 0.4^2 \}, \\
\mathrm{rectangle} &= \{(x_1,x_2)^\top \in \R^2: |x_1-0.6| < 0.45, |x_2+0.6| < 0.25\}, \\
D &= \mathrm{disk} \cup \mathrm{rectangle},  \\ 
\eta(x)  &= 0.5 \quad \text{in } D.
\end{align*}
c) Square-shaped  object with  cavity
\begin{align*}
\mathrm{cavity} &= \{ (x_1,x_2)^\top \in \R^2: x_1^2 + x_2^2 < 0.3^2 \}, \\
\mathrm{square} &= \{ (x_1,x_2)^\top \in \R^2: |x_1| < 0.5, |x_2| < 0.5\}, \\
D &= \mathrm{square} \setminus \mathrm{\overline{cavity}},  \\ 
\eta(x)  &= 1 \quad \text{in } D.
\end{align*}
 To generate the scattering data for the numerical examples, we solve the Lippmann-Schwinger equation~\eqref{LS}
using a  spectral Galerkin  method   developed in~\cite{Lechl2014}. 
Using $N_d$ incident plane waves and measuring the data at $N_x$ points on $\partial \Omega$, 
the Cauchy data $u_\sc(x,d), \partial u_\sc/ \partial \nu(x,d)$, where $(x,d) \in \partial \Omega \times \S$, are then $N_d \times N_x$ matrices. The artificial noise is added to the data as follows.
Two complex-valued noise matrices $\mathcal{N}_{1,2}$ 
containing random numbers that are uniformly distributed in the complex square 
$$\{ a+ i b\, : \, |a| \leq 1, \, |b| \leq 1 \} \subset \C$$
 are added to the data matrices. For simplicity we consider the same noise level  $\delta$ for both $u_\sc$ and $\partial u_\sc/ \partial \nu$.  
The  noisy data $u_\sc^\delta$ and $\partial u^\delta_\sc/ \partial \nu$ are given by
\begin{align*}
u_\sc^\delta 
 := u_\sc 
   + \delta\frac{\mathcal{N}_1}{\|\mathcal{N}_1\|_F} \|u_\sc  \|_F, \quad 
\frac{\partial u_\sc^\delta }{\partial \nu} := \frac{\partial u_\sc }{\partial \nu}  
   + \delta\frac{\mathcal{N}_2}{\|\mathcal{N}_2\|_F} \left\|\frac{\partial u_\sc }{\partial \nu} \right \|_F,
\end{align*}
where $\|\cdot\|_F$ is the Frobenius matrix norm.

\subsection{Reconstruction with different wave numbers (Figure \ref{fi1})}
We present in Figure~\ref{fi2}  reconstruction results for the wave numbers $k = 4$ (wavelength $\approx$ 1.57) 
and $k = 8$ (wavelength $\approx$ 0.78). The data 
are near-field Cauchy data with 30$\%$ noise. We use $N_x\times N_d = 64\times 64$  for the kite-shaped  object
and disk-and-rectangle  object, while the square-shaped object with cavity is examined with 
$N_x\times N_d = 96\times 96$. It can be seen from the Figure~\ref{fi1} that the reconstruction results 
are improved with a larger value of $k$. We also see that the two imaging functionals can image very well the square-shaped 
object with cavity. It is interesting that this object violates the assumption that $\R^n \setminus \overline{D}$ must be connected when studying the well-posedness of the direct scattering problem~\eqref{Helm}--\eqref{radiation}.

\begin{figure}[h!!!]
\centering
\subfloat[True geometry]{\includegraphics[width=4cm]{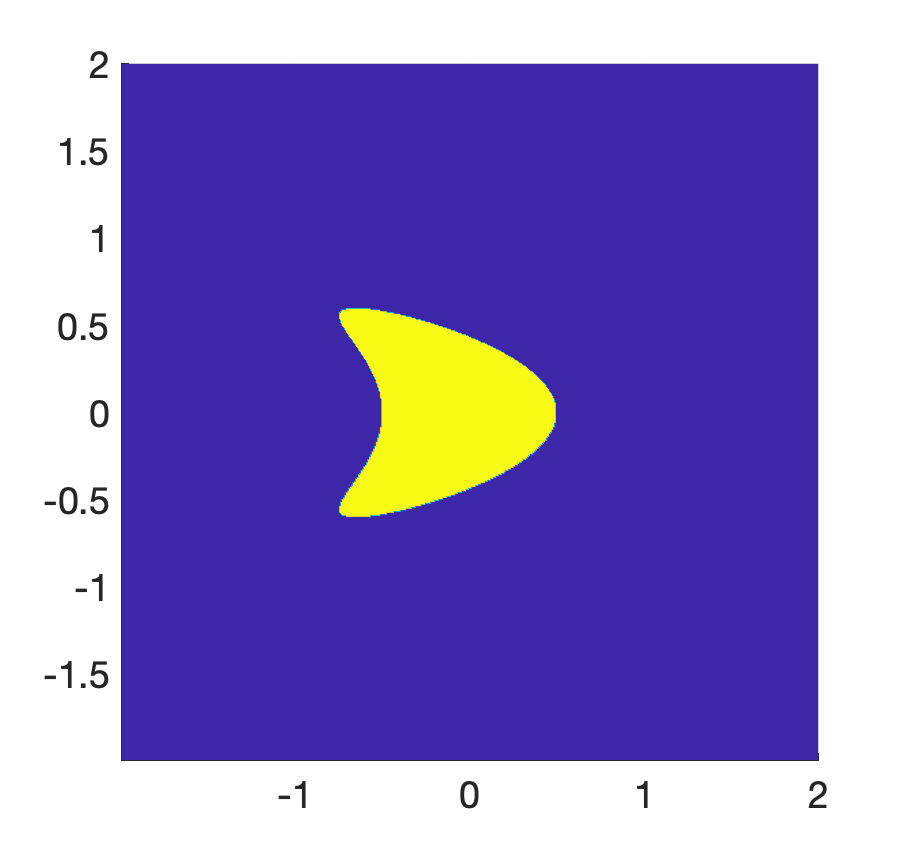}} \hspace{0cm}
\subfloat[$k =4$]{\includegraphics[width=4.5cm]{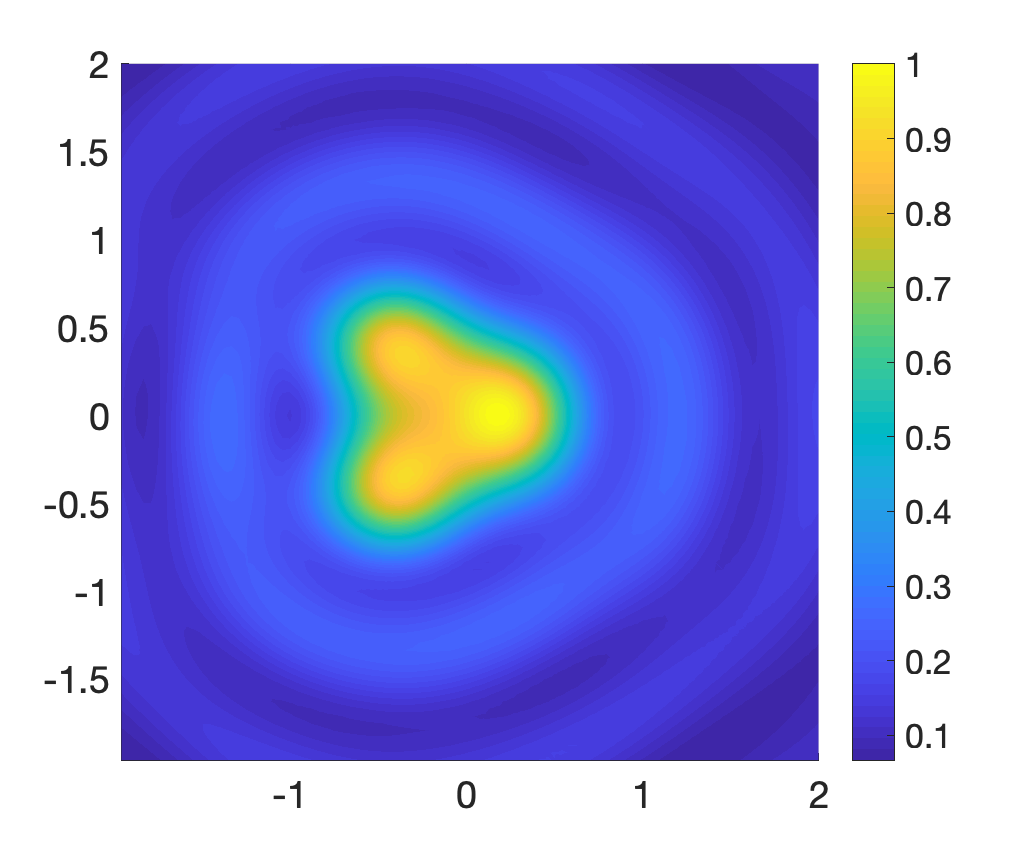}}  \hspace{0cm} 
\subfloat[$ k =8$]{\includegraphics[width=4.5cm]{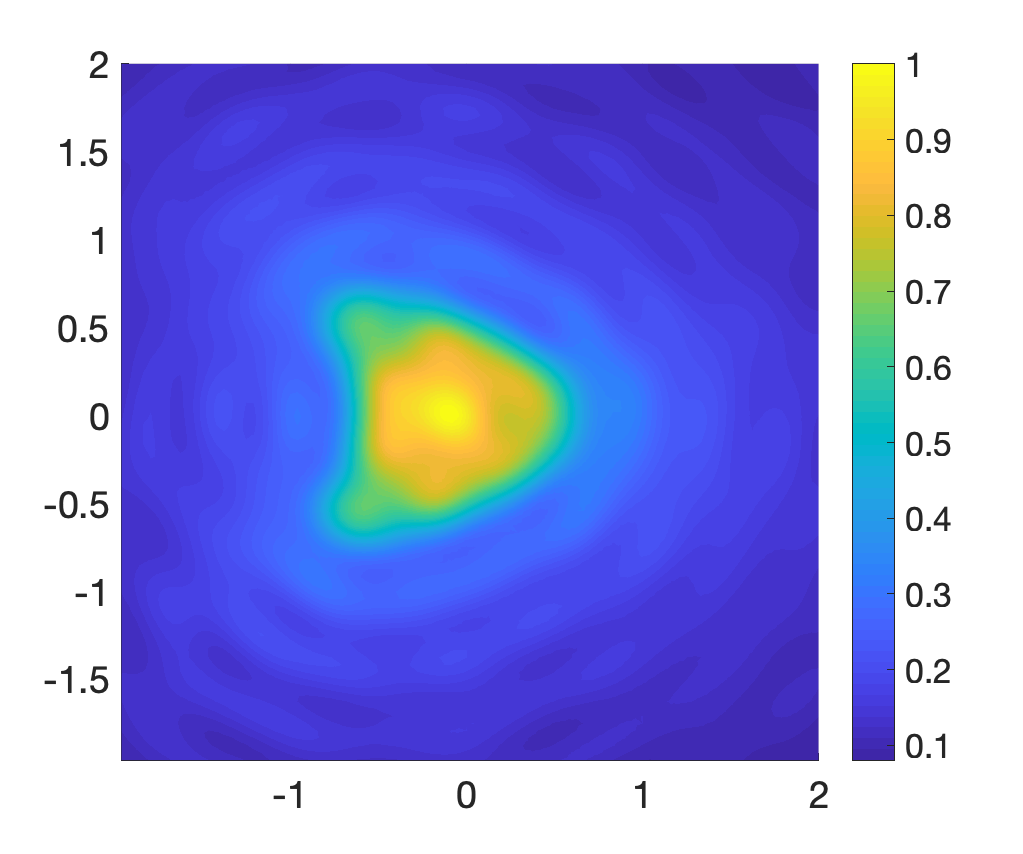}}  \hspace{0cm}\\
\subfloat[True geometry]{\includegraphics[width=4cm]{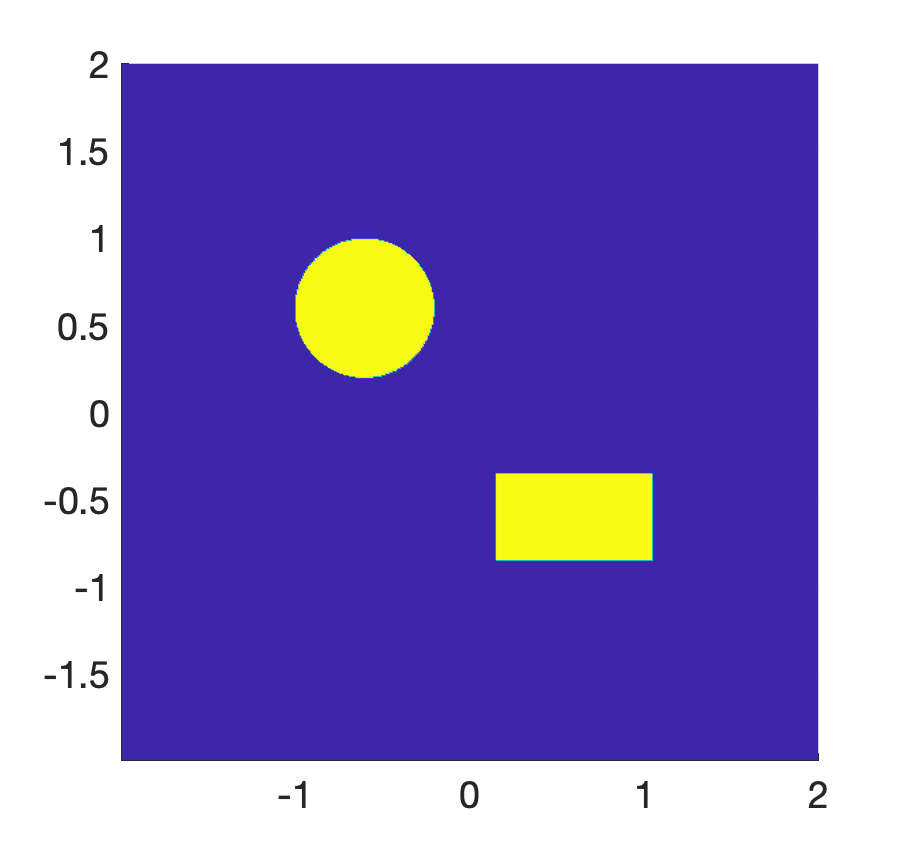}} \hspace{0cm}
\subfloat[$ k =4$]{\includegraphics[width=4.5cm]{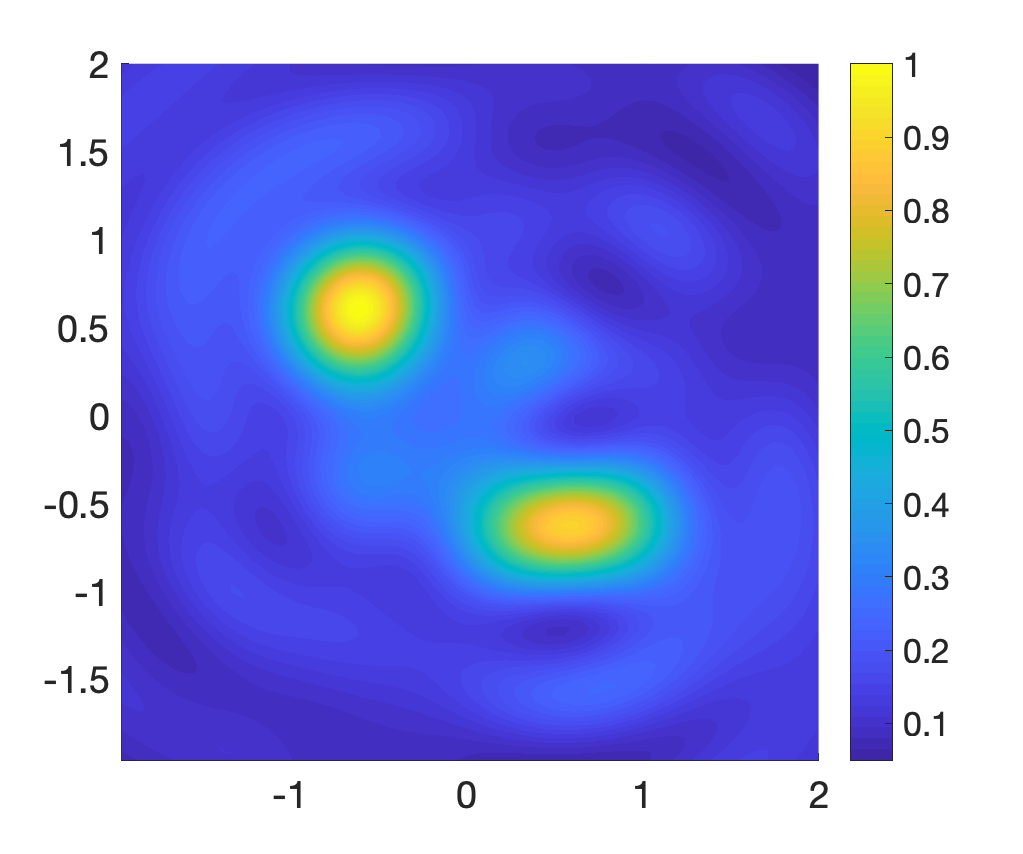}}  \hspace{0cm} 
\subfloat[$k =8$]{\includegraphics[width=4.5cm]{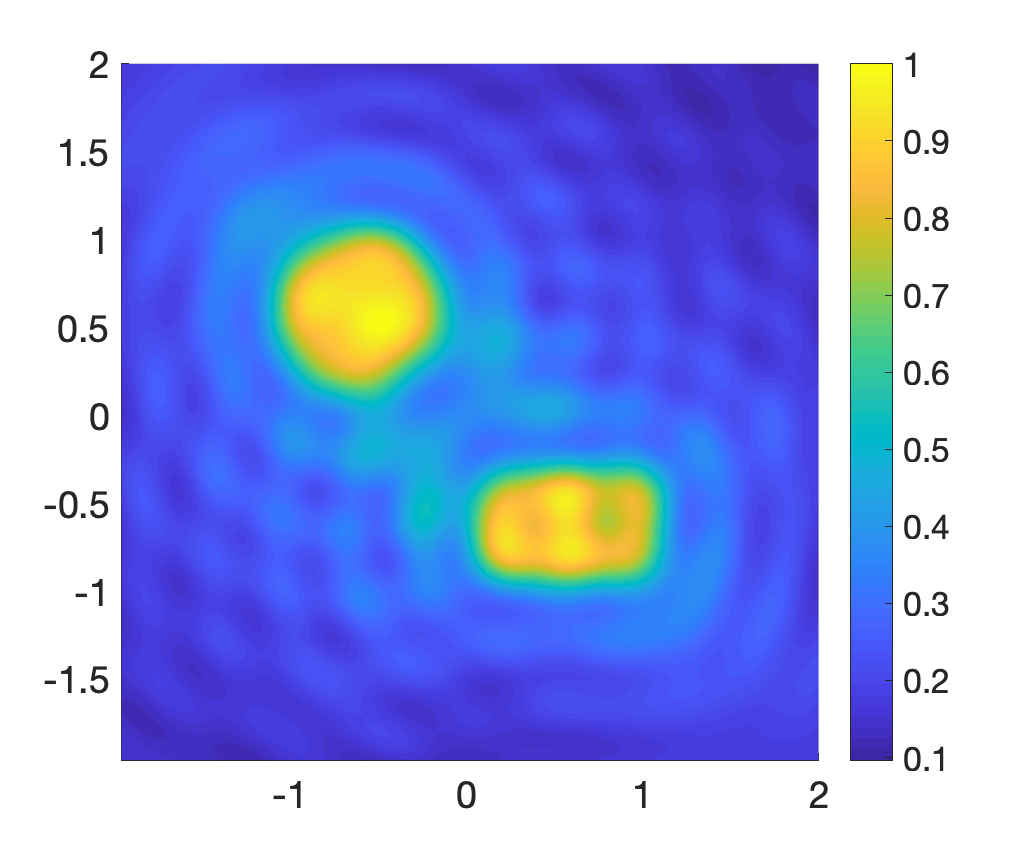}}  \hspace{0cm}\\
\subfloat[True geometry]{\includegraphics[width=4cm]{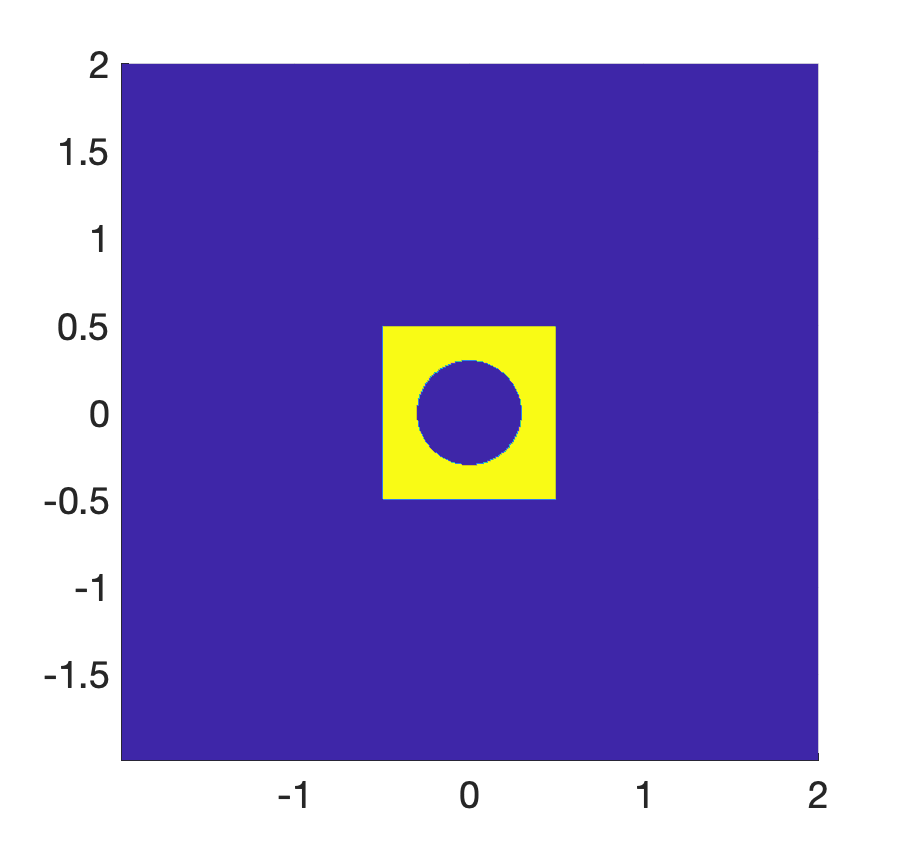}} \hspace{0cm}
\subfloat[$ k =4$]{\includegraphics[width=4.5cm]{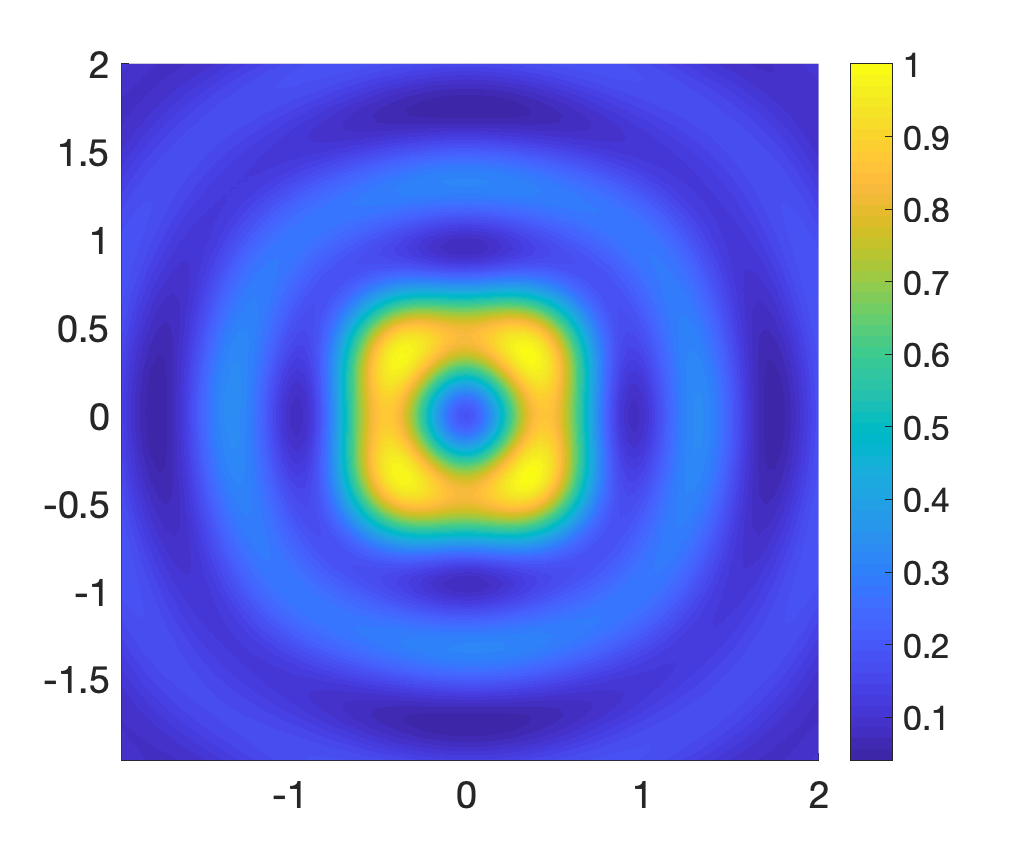}}  \hspace{0cm} 
\subfloat[$k =8$]{\includegraphics[width=4.5cm]{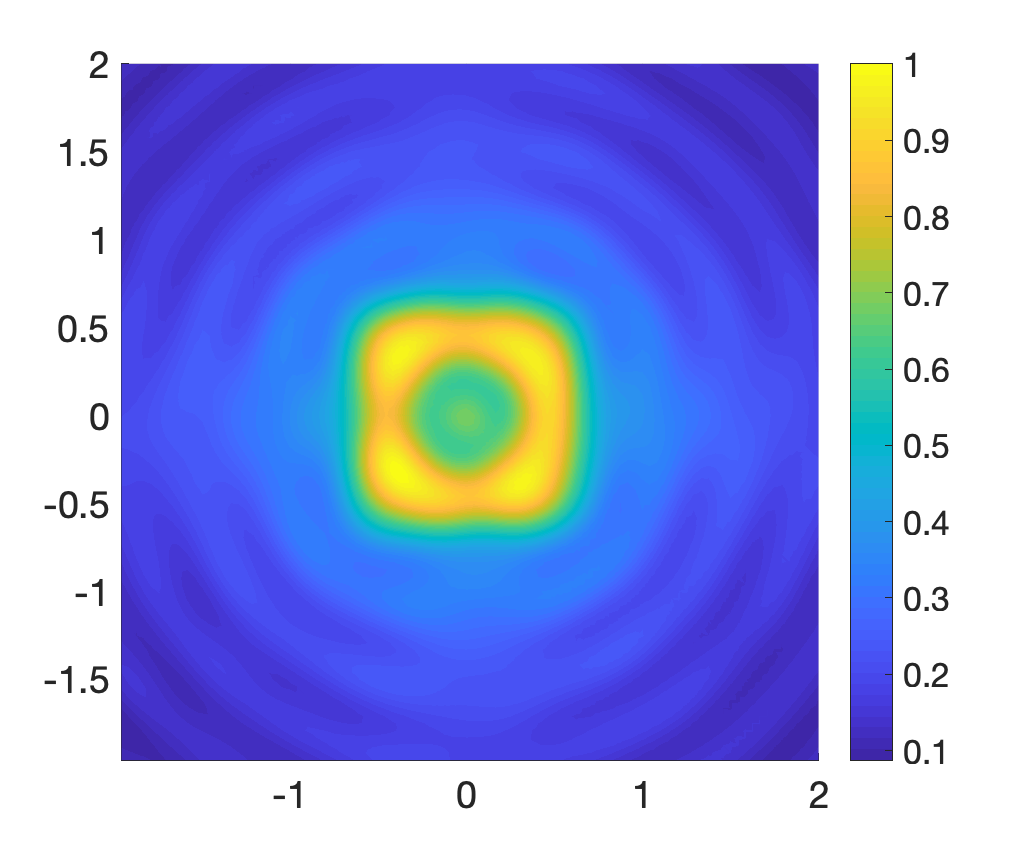}}  \hspace{0cm}
\caption{Reconstruction with near-field Cauchy data for different wave numbers.  
There is  30$\%$ noise added to the data ($\delta  =0.3$).
First column (a, d, g): true geometry. Second column (b, e, h): reconstruction with $k=4$.
Third column (c, f, i): reconstruction with  $k=8$.
 } 
 \label{fi1}
\end{figure}

\subsection{Reconstruction with highly noisy data (Figure \ref{fi2})}
We present in Figure~\ref{fi2}  reconstruction results for near-field Cauchy data perturbed by $60\%$ and $90\%$ 
noise. The wave number $k = 8$ and again we use $N_x\times N_d = 64\times 64$  for the kite-shaped  object
and disk-and-rectangle  object,  and $N_x\times N_d = 96\times 96$ for  the square-shaped object with cavity.
Although we can notice some deterioration in the case of $90\%$ noise, the reconstructions are still
pretty reasonable. These results show that the sampling method is extremely robust with respect to noise
in the data.  We have also observed this robustness in the orthogonality sampling method for Maxwell's equations, see~\cite{Harri2020}.

\begin{figure}[h!!!]
\centering
\subfloat[True geometry]{\includegraphics[width=4cm]{kite_true}} \hspace{0cm}
\subfloat[ 60$\%$ noise]{\includegraphics[width=4.5cm]{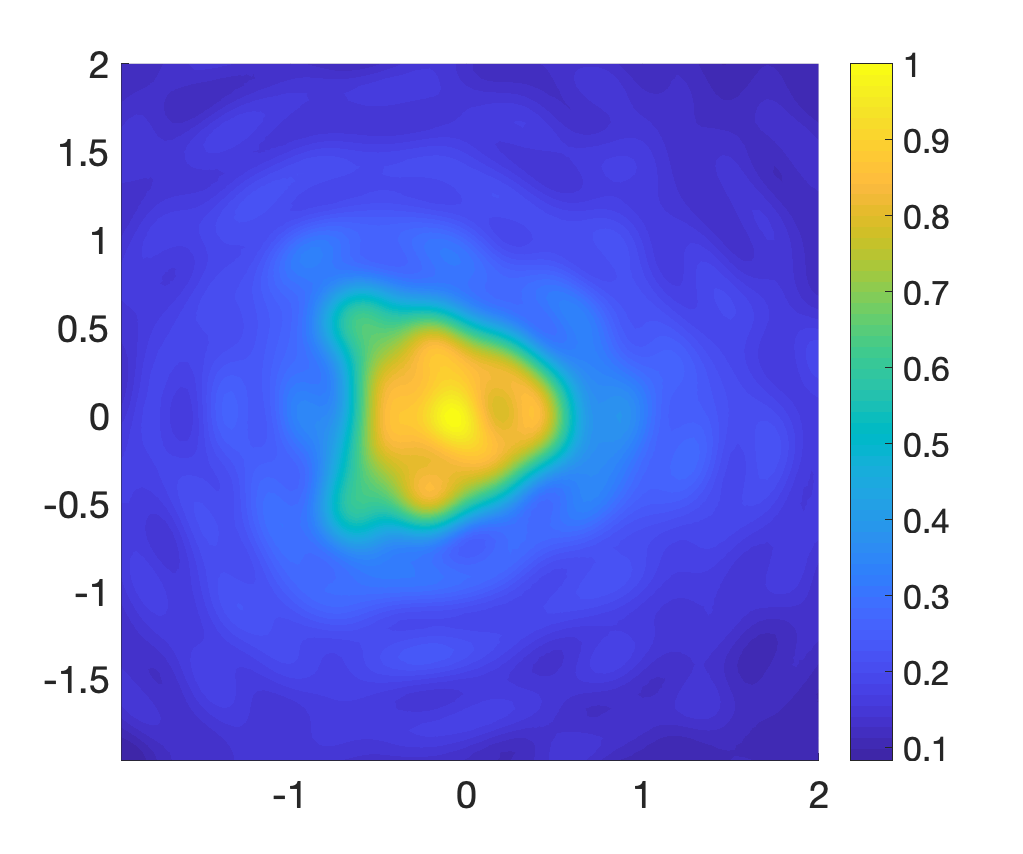}}  \hspace{0cm} 
\subfloat[90$\%$ noise]{\includegraphics[width=4.5cm]{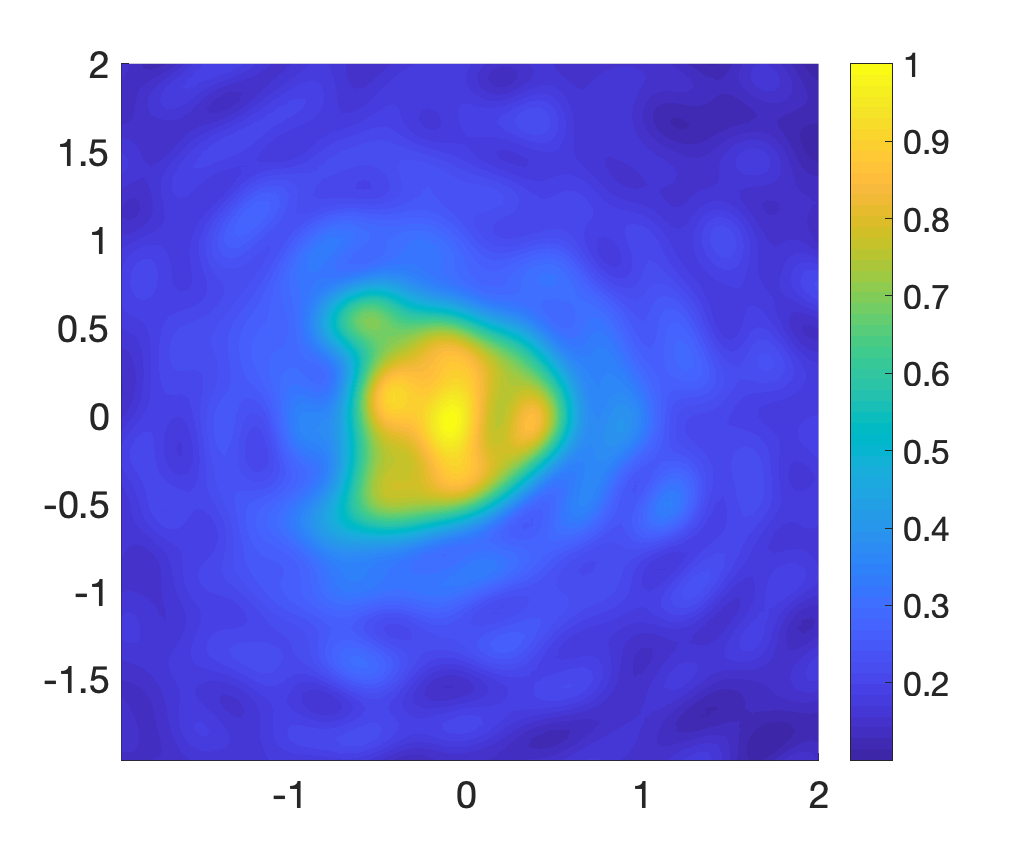}}  \hspace{0cm}\\
\subfloat[True geometry]{\includegraphics[width=4cm]{disk_rec_true}} \hspace{0cm}
\subfloat[60$\%$ noise]{\includegraphics[width=4.5cm]{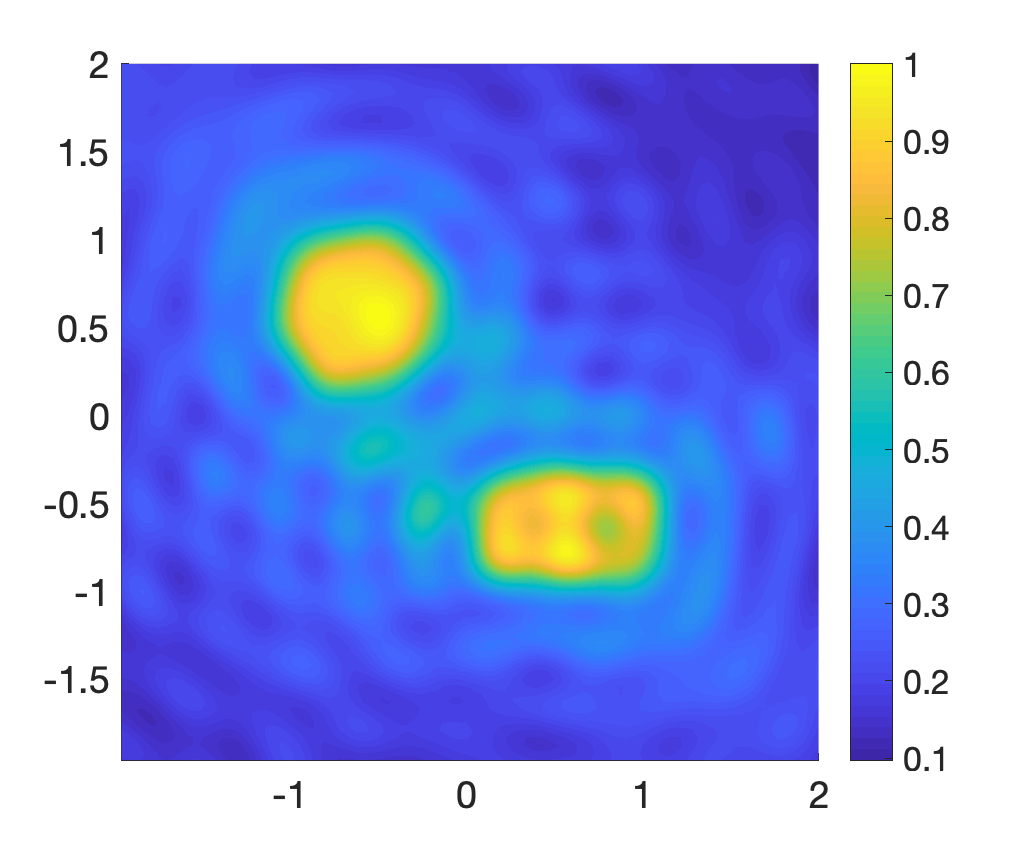}}  \hspace{0cm} 
\subfloat[90$\%$ noise]{\includegraphics[width=4.5cm]{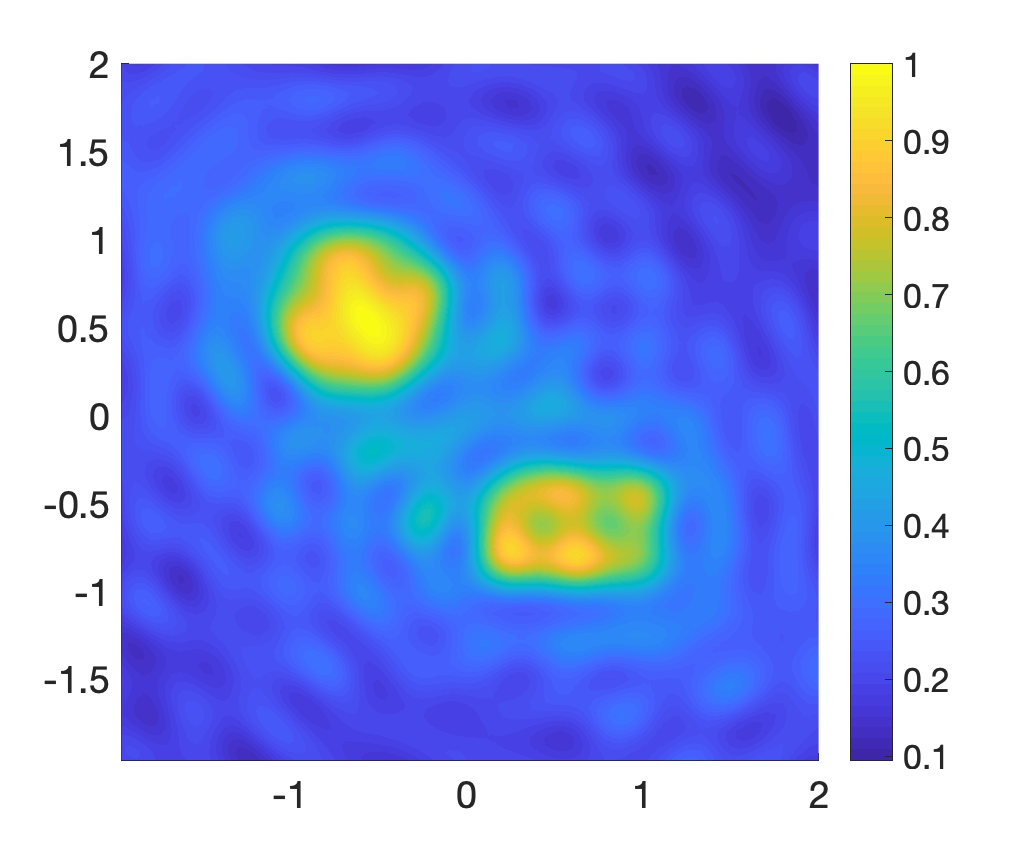}}  \hspace{0cm} \\
\subfloat[True geometry]{\includegraphics[width=4cm]{rechole_true}} \hspace{0cm}
\subfloat[60$\%$ noise]{\includegraphics[width=4.5cm]{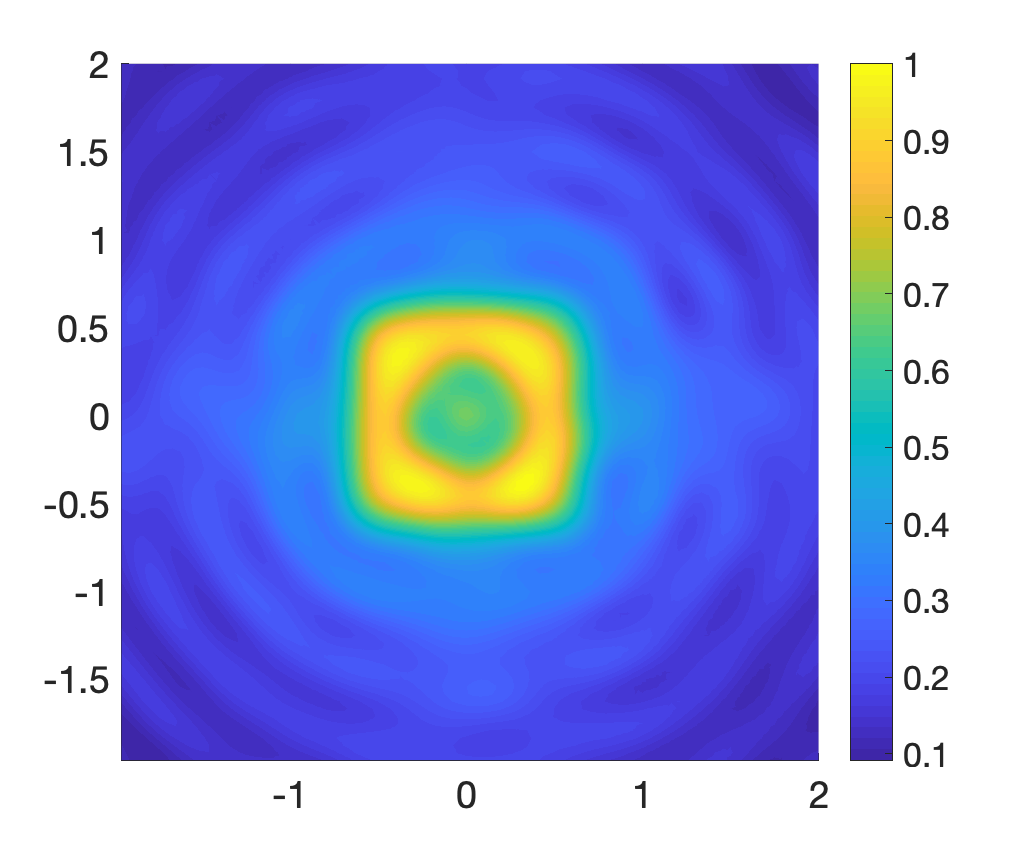}}  \hspace{0cm} 
\subfloat[90$\%$ noise]{\includegraphics[width=4.5cm]{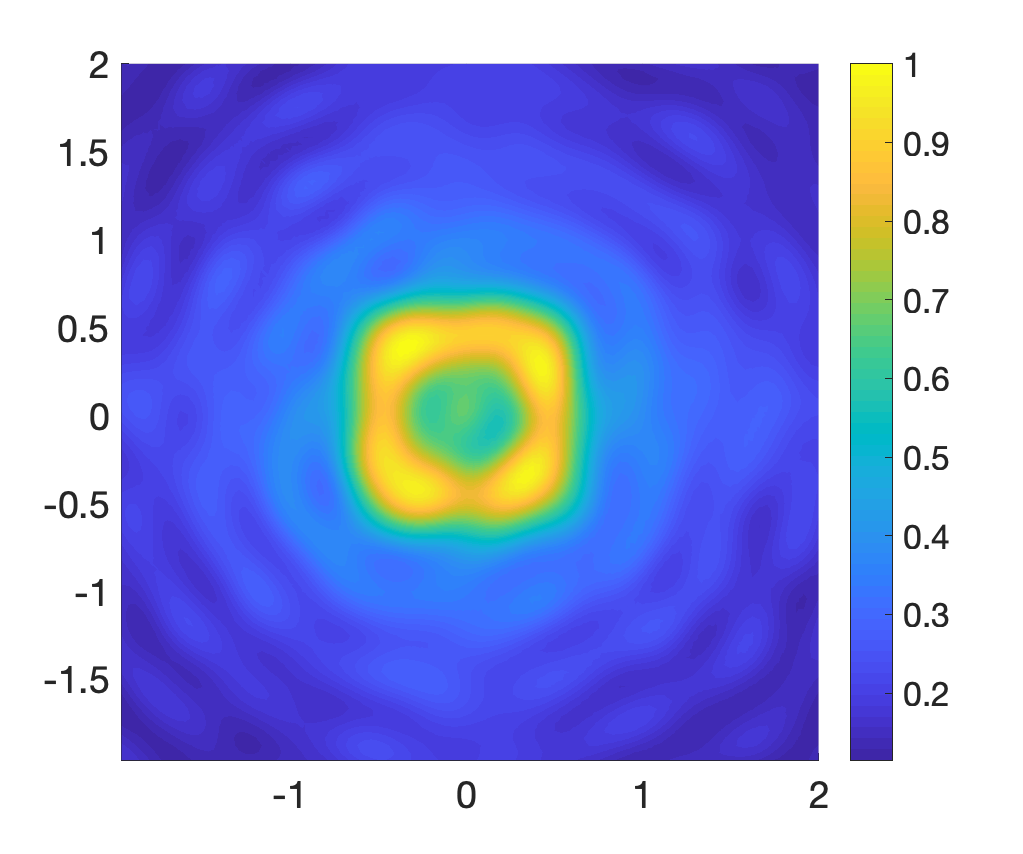}}  \hspace{0cm}
\caption{Reconstruction  with highly noisy near-field data. Wave number $k = 8$.  
First column (a, d, g): true geometry. Second column (b, e, h): reconstruction with $60\%$ noise.
Third column (c, f, i): reconstruction with $90\%$ noise.
 } 
 \label{fi2}
\end{figure}

\subsection{Reconstruction with far-field  data (Figure \ref{fi3})}
The focus of this example is to examine the performance of the sampling methods associated with $I(z)$ and $I_{\mathrm{far}}(z)$, defined by~\eqref{Ifar},  in the case of far-field data with $30\%$ noise. We recall that  
$I_{\mathrm{far}}(z)$  uses only $u_\sc$ instead of the Cauchy data. Again we consider $k = 8$
and the size $N_x\times N_d$ of the data matrices are the same as in the previous examples. As mentioned at the beginning of this section the far-field  data are measured on $\partial \Omega$ that is the circle of radius $R = 100$ (about 125 wavelengths away from the scattering objects). We can see in Figure~\ref{fi3} that the reconstruction results 
with far-field data are as good as those with  the near-field data. The two imaging functionals $I(z)$ and  $I_{\mathrm{far}}(z)$ provide similar results as expected.

\begin{figure}[ht!]
\centering
\subfloat[True geometry]{\includegraphics[width=4cm]{kite_true}} \hspace{-0.0cm}
\subfloat[$I(z)$]{\includegraphics[width=4.5cm]{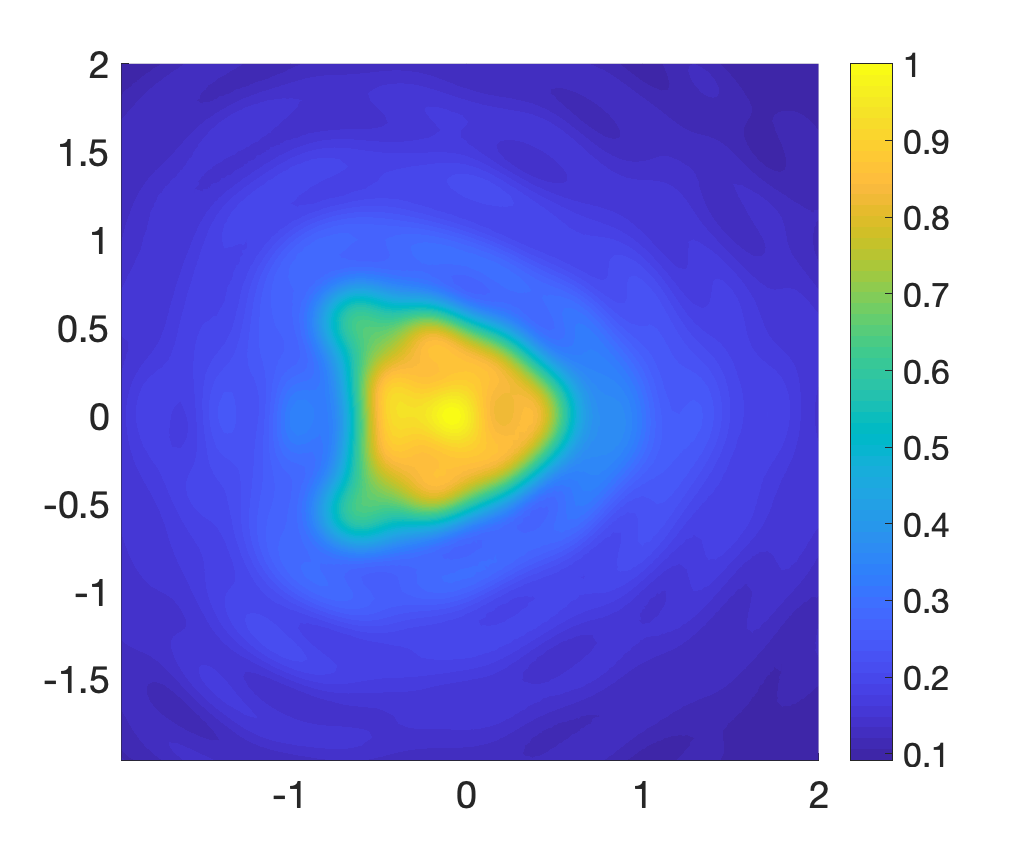}}  \hspace{-0.0cm} 
\subfloat[$I_{\mathrm{far}}(z)$]{\includegraphics[width=4.5cm]{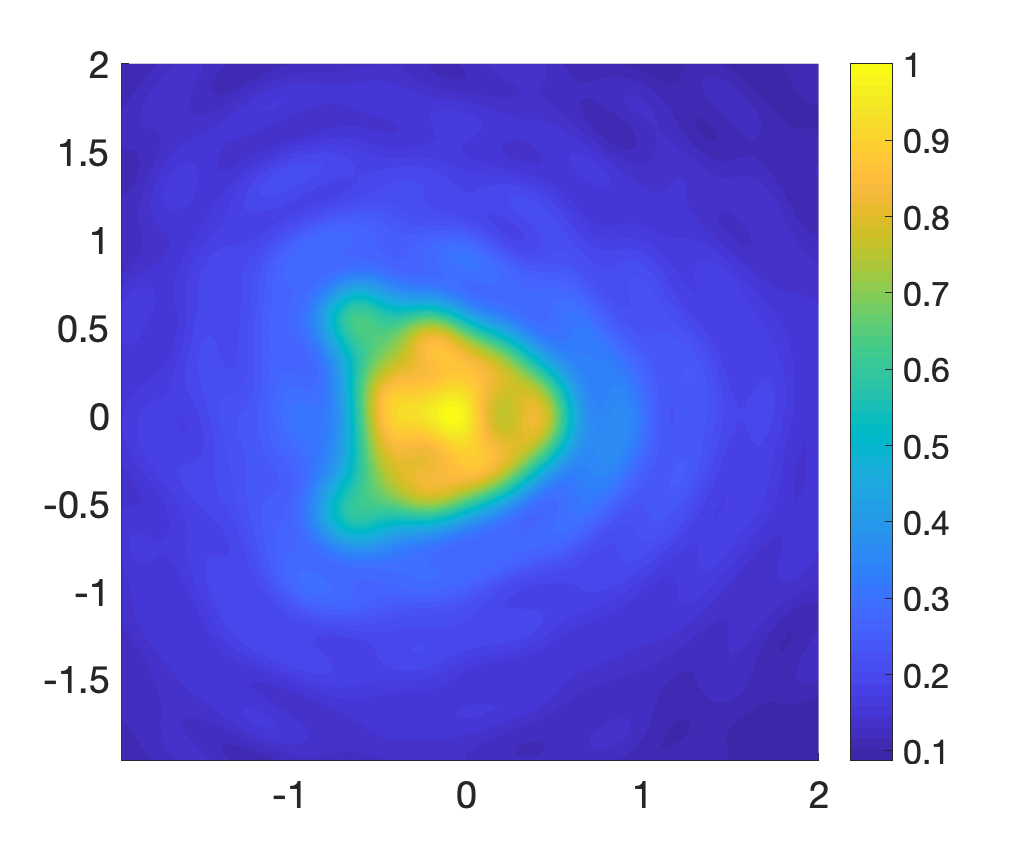}}  \hspace{-0.0cm}\\
\subfloat[True geometry]{\includegraphics[width=4cm]{disk_rec_true}} \hspace{-0.0cm}
\subfloat[$I(z)$]{\includegraphics[width=4.5cm]{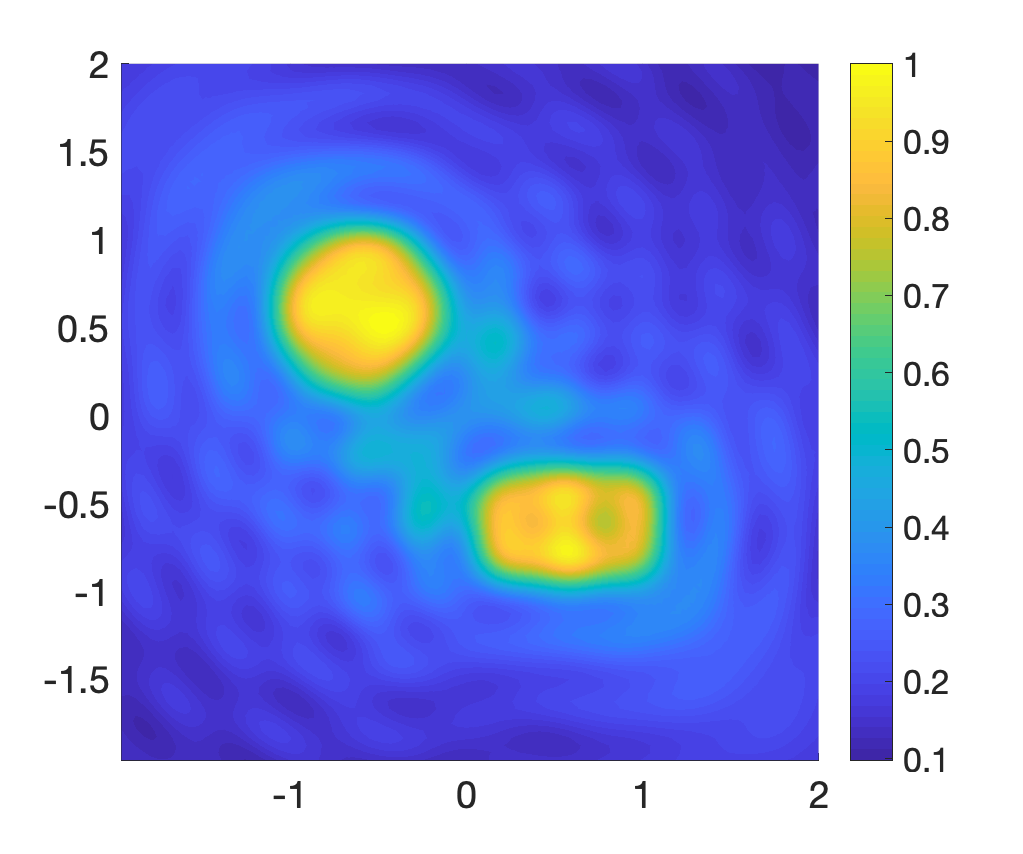}}  \hspace{-0.0cm} 
\subfloat[$I_{\mathrm{far}}(z)$]{\includegraphics[width=4.5cm]{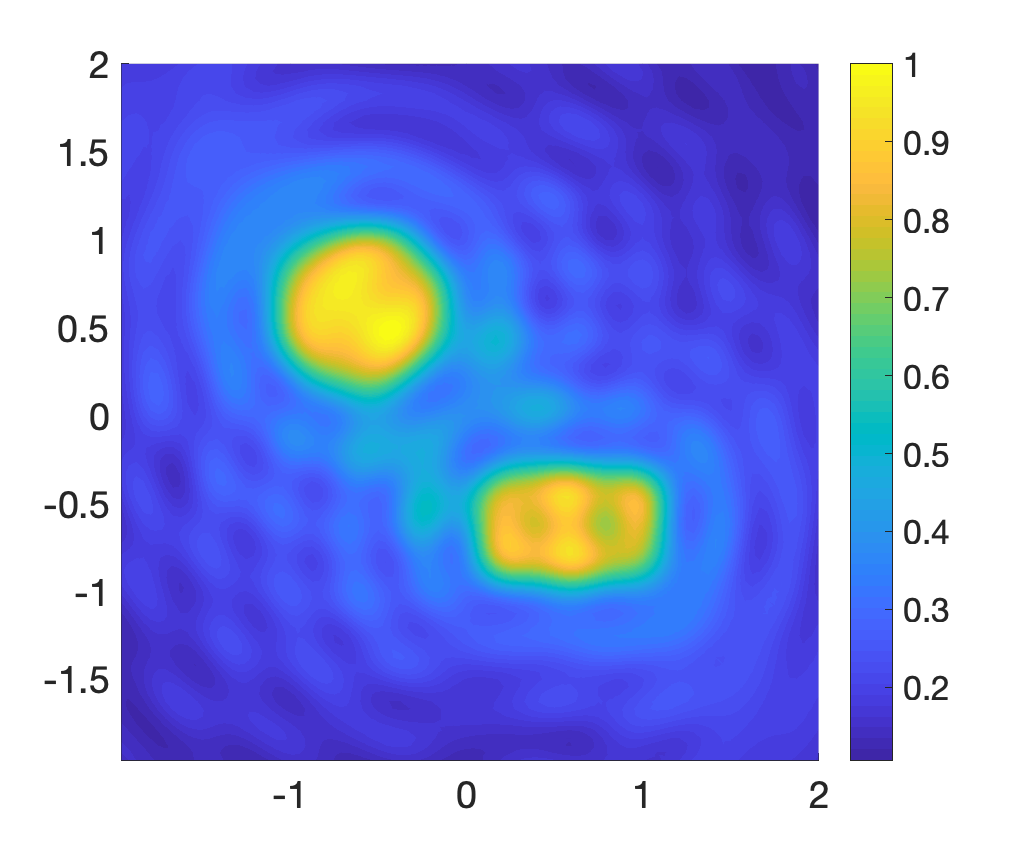}}  \hspace{-0.0cm}\\
\subfloat[True geometry]{\includegraphics[width=4cm]{rechole_true}} \hspace{-0.0cm}
\subfloat[$I(z)$]{\includegraphics[width=4.5cm]{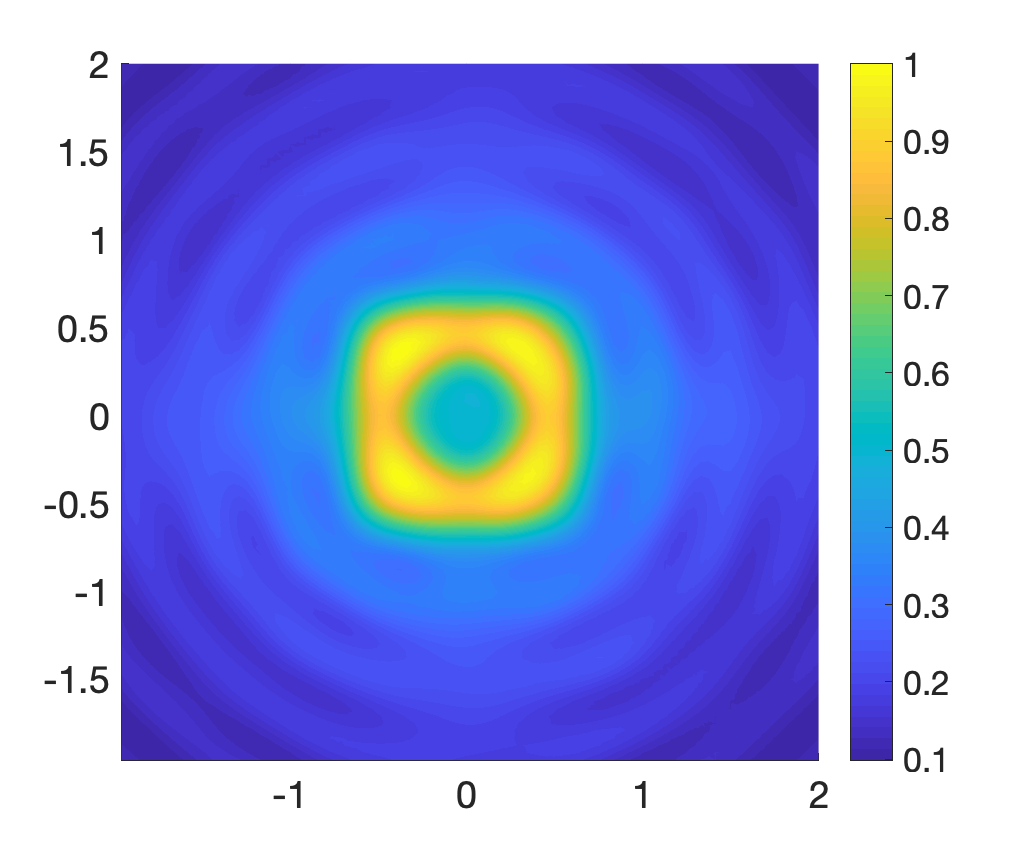}}  \hspace{-0.0cm} 
\subfloat[$I_{\mathrm{far}}(z)$]{\includegraphics[width=4.5cm]{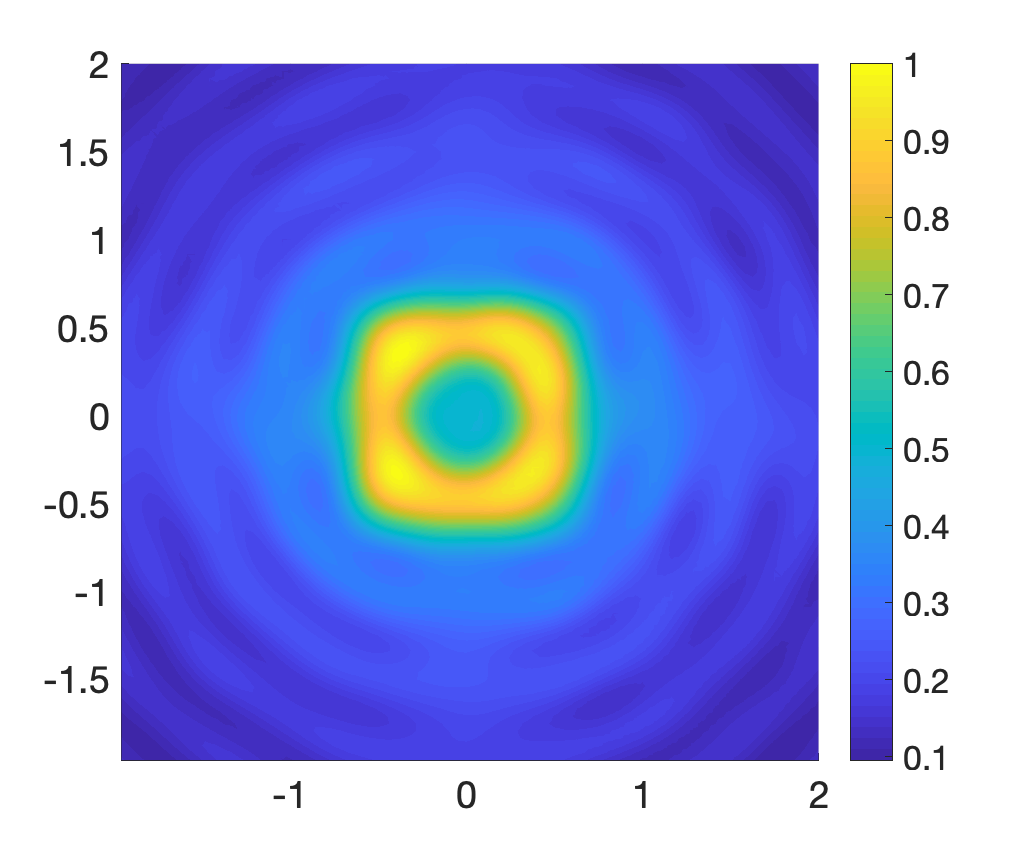}}  \hspace{-0.0cm}
\caption{Reconstruction with far-field data.  
There is  30$\%$ noise added to the data, $k = 8$. First column (a, d, g): true geometry. Second column (b, e, h): reconstruction   using $I(z)$.
Third column (c, f, i): reconstruction using $I_{\mathrm{far}}(z)$.
 } 
 \label{fi3}
\end{figure}

\subsection{Reconstruction with limited aperture  data (Figure \ref{fi4})}
In this last  example we consider near-field  data  for a half-circle aperture  ($30\%$ noise). More precisely,
 the incident point sources   are located on  the upper half 
 the measurement circle $\partial \Omega$, and the Cauchy data are  given on the bottom half of  $\partial \Omega$. Moreover, the number of data points $N_x$ and incident plane waves $N_d$ are also half of those of  the full data case, that means $N_x \times N_d = 32\times 32$ for the kite-shaped object and disk-and-rectangle object, and   $N_x \times N_d = 48\times 48$ for the square-shaped object with cavity.  As it can be seen from Figure~\ref{fi4}, the  reconstruction results  for the the first two objects are still pretty reasonable. However, the shape of the reconstructed square-shaped object with cavity is no longer accurate. This object is certainly more difficult to image compared with the first two objects.

\begin{figure}[ht!]
\centering
\subfloat[True geometry]{\includegraphics[width=4.5cm]{kite_true}} \hspace{0.5cm}
\subfloat[True geometry]{\includegraphics[width=4.5cm]{disk_rec_true}} \hspace{0.5cm}
\subfloat[True geometry]{\includegraphics[width=4.5cm]{rechole_true}} \hspace{0.5cm}\\
\subfloat[]{\includegraphics[width=4.8cm]{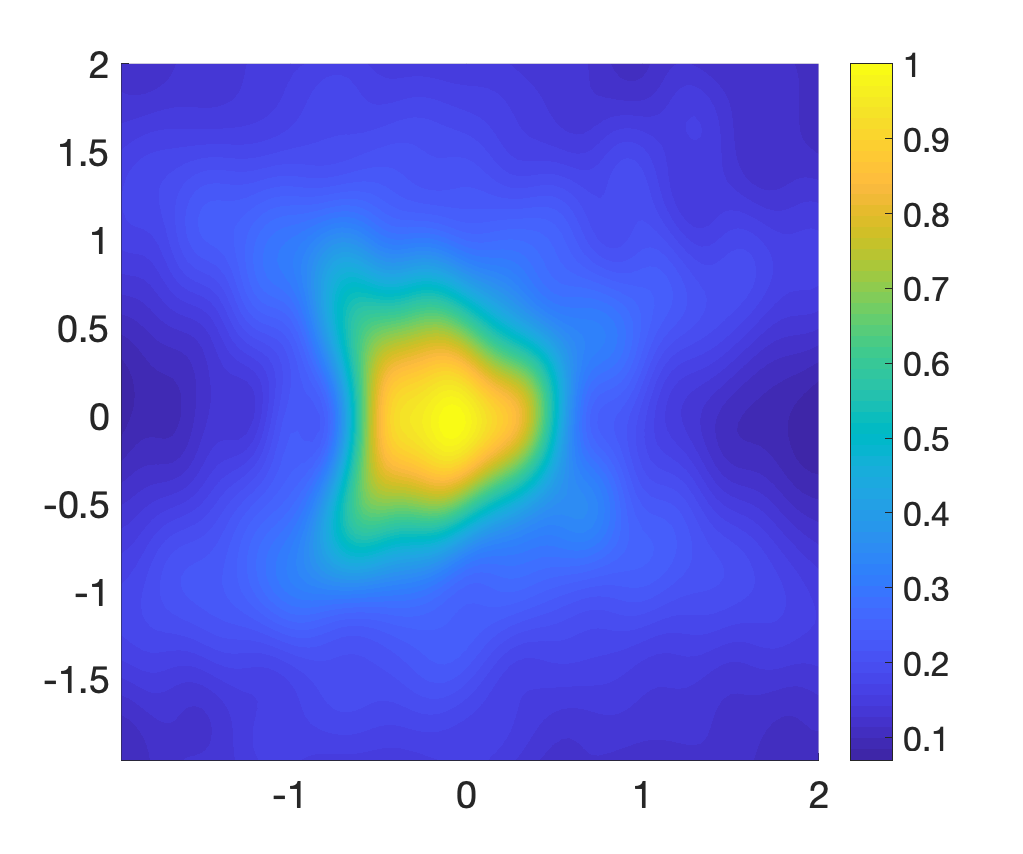}} \hspace{0.3cm} 
\subfloat[]{\includegraphics[width=4.8cm]{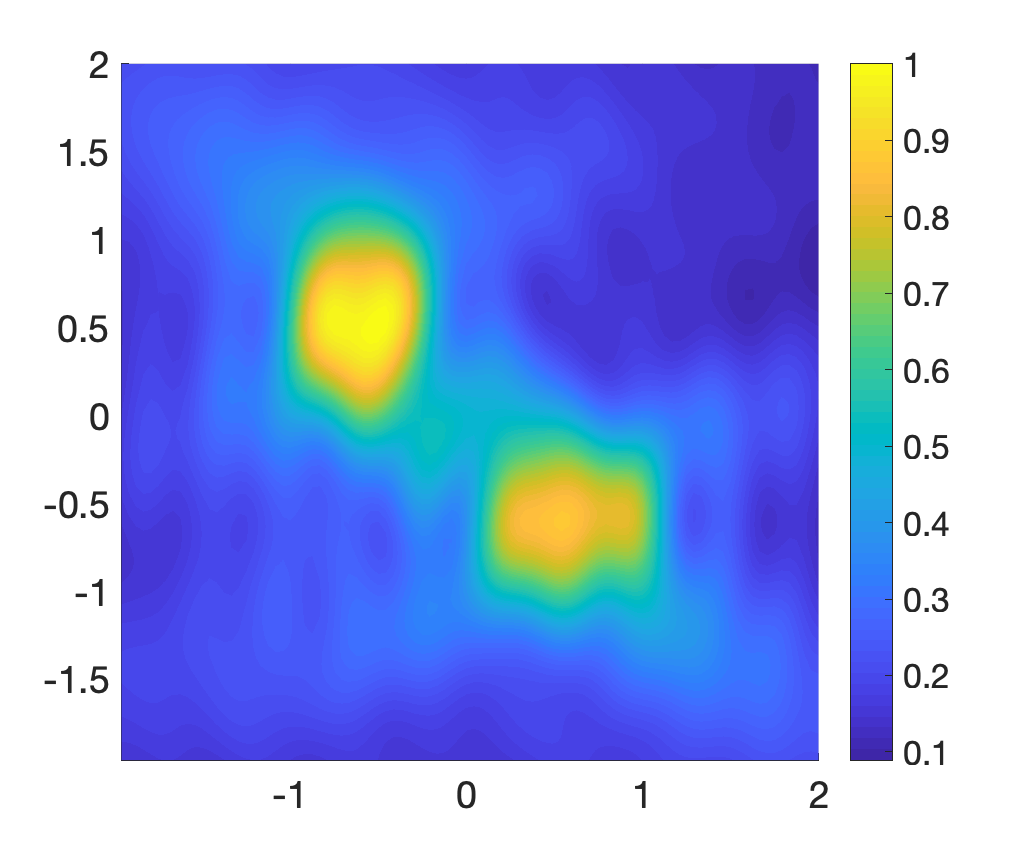}} \hspace{0.3cm} 
\subfloat[]{\includegraphics[width=4.8cm]{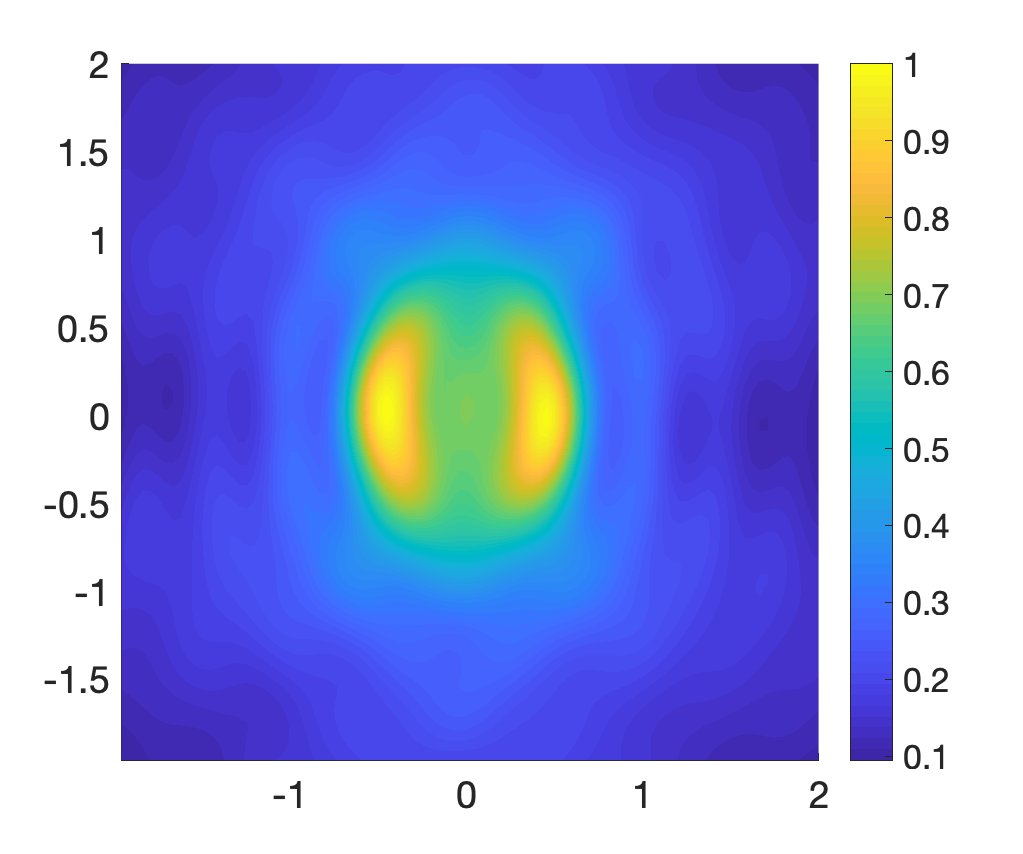}}  \hspace{-0.3cm} 
%\subfloat[]{\includegraphics[width=4.2cm]{rechole_k8_d96_r100_usc_I2}}  
\caption{Reconstruction with partial limited aperture near-field data.  
There is  30$\%$ noise added to the data and $k = 8$.
 } 
 \label{fi4}
\end{figure}

\subsection{Reconstruction with experimental  data (Figure~\ref{fi5})}
In this section we verify the performance of the indicator function with experimental data provided by Institut Fresnel (France). 
We used the data sets of homogeneous objects. Three data sets were investigated: the first one named \textit{dielTM\_dec4f.exp} is associated with a dielectric target  that is a de-centered circular cross section of radius $15$ mm, and the second one is \textit{rectTM\_cent.exp} concerning a centered rectangular cross section (dielectric material) of dimensions $25.4 \times 12.7$ mm$^2$. The last one named \textit{uTM\_shaped.exp}  is associated with a metallic U-shaped object of size 50  $\times$ 80 mm$^2$. A detailed description of the experimental setup can be found in \cite{Belke2001}.  

We rescaled 40 mm to be 1 unit of length in our MATLAB simulations.  Measurement distance from the origin is about 0.76 m which is close to 19  in our simulation. The data are clearly measured in a far-field regime. The data matrix size is 72 (receivers) $\times$ 36 (incident sources), where 72 receivers are distributed at the angular range from $60^\circ$ to $300^\circ$ in steps of $5^\circ$ and  the rotation of the target for the source is from $0^\circ$ to $350^\circ$ in steps of $10^\circ$. For the convenience of the readers we create the  geometry of these targets in Figures~\ref{fi5}(a, b, c) so that we can compare with the reconstruction results.

We consider the wave frequency 8 GHz for the data sets (wave number $k$ is about 6.7 which means the wavelength is about 0.93). Since we only have the scattered wave data  on a circle in a far field regime, we use $I_{\text{far}}(z)$ to reconstruct the targets.  We compute $I_{\text{far}}(z)$ at 64 $\times$ 64 sampling points in the search domain $(-2.5, 2.5)^2$.  There is no need for any regularization  or any further processing for the experimental data.  In Figures~\ref{fi5} we can see that the indicator function $I_{\text{far}}(z)$ is able to reconstruct the targets with reasonable accuracy.

\begin{figure}[ht!]
\centering
\subfloat[True geometry]{\includegraphics[width=4.5cm]{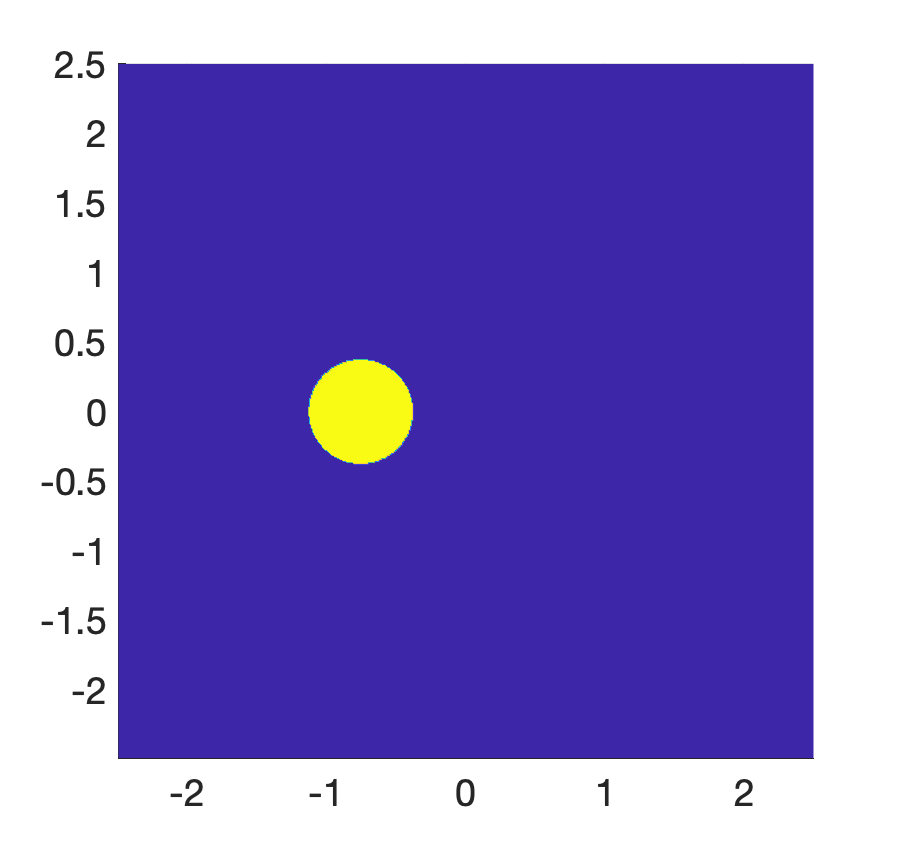}} \hspace{0.5cm}
\subfloat[True geometry]{\includegraphics[width=4.5cm]{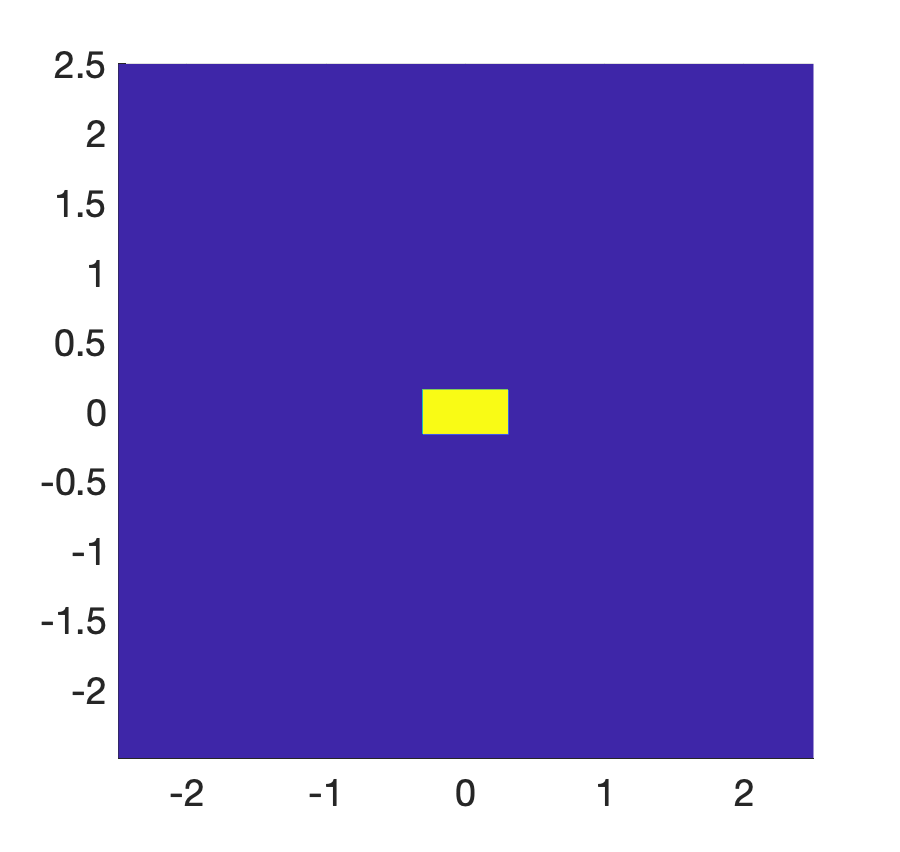}} \hspace{0.5cm}
\subfloat[True geometry]{\includegraphics[width=4.5cm]{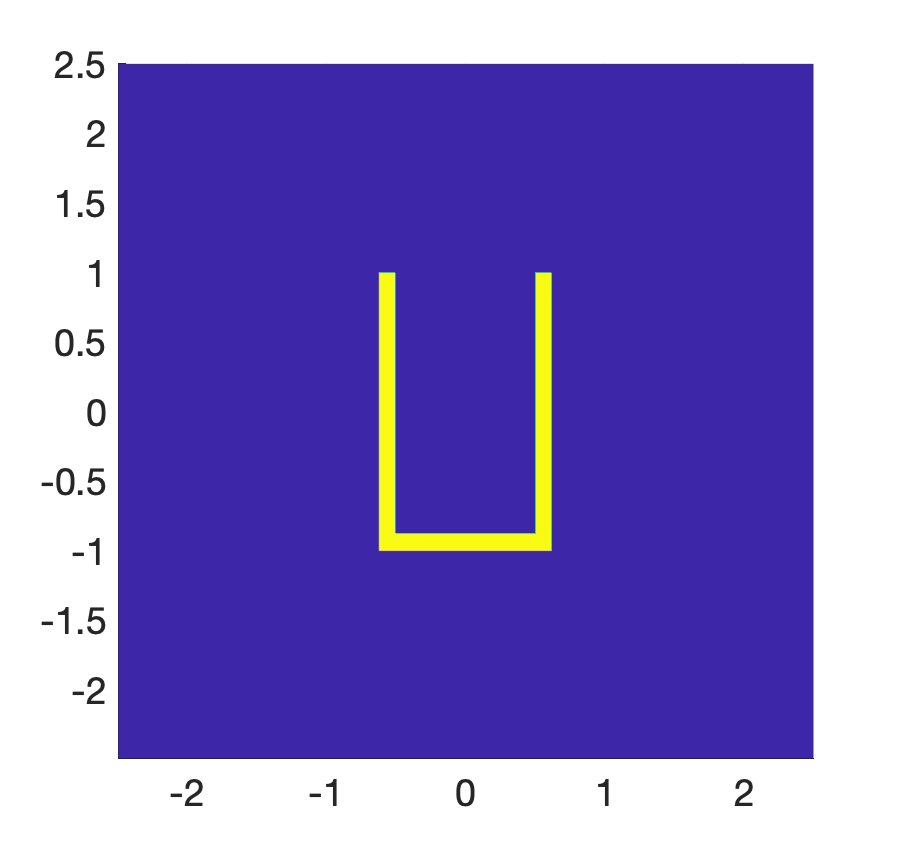}} \hspace{0.5cm} \\
\subfloat[]{\includegraphics[width=4.8cm]{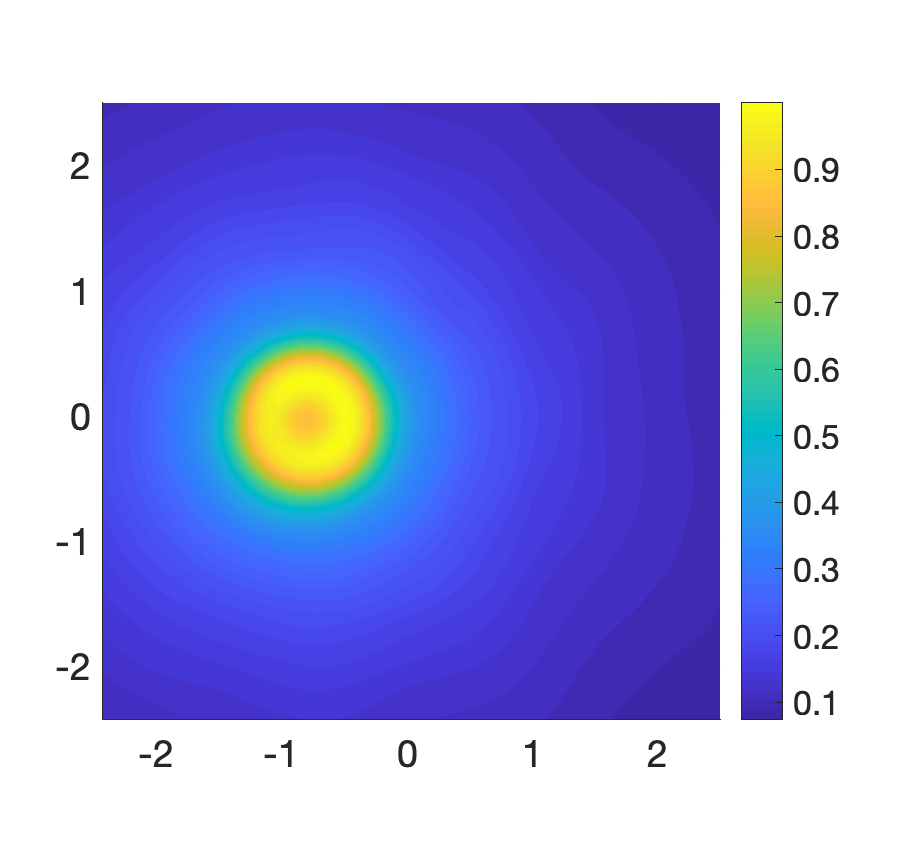}} \hspace{0.3cm}
\subfloat[]{\includegraphics[width=4.8cm]{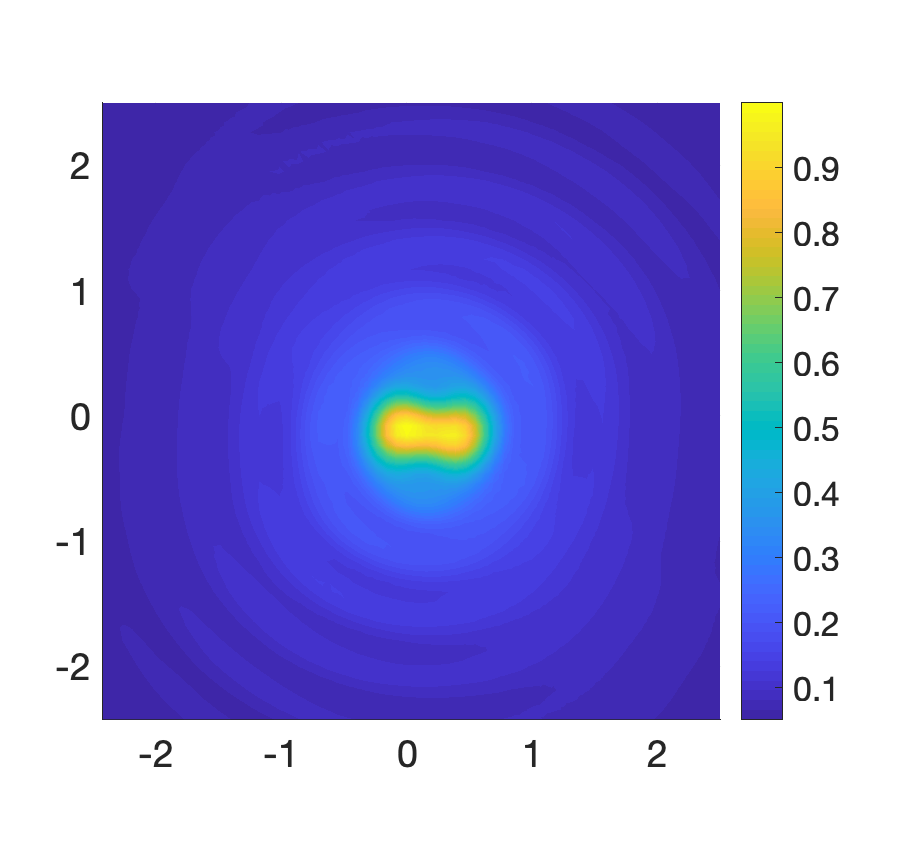}} \hspace{0.3cm} 
\subfloat[]{\includegraphics[width=4.8cm]{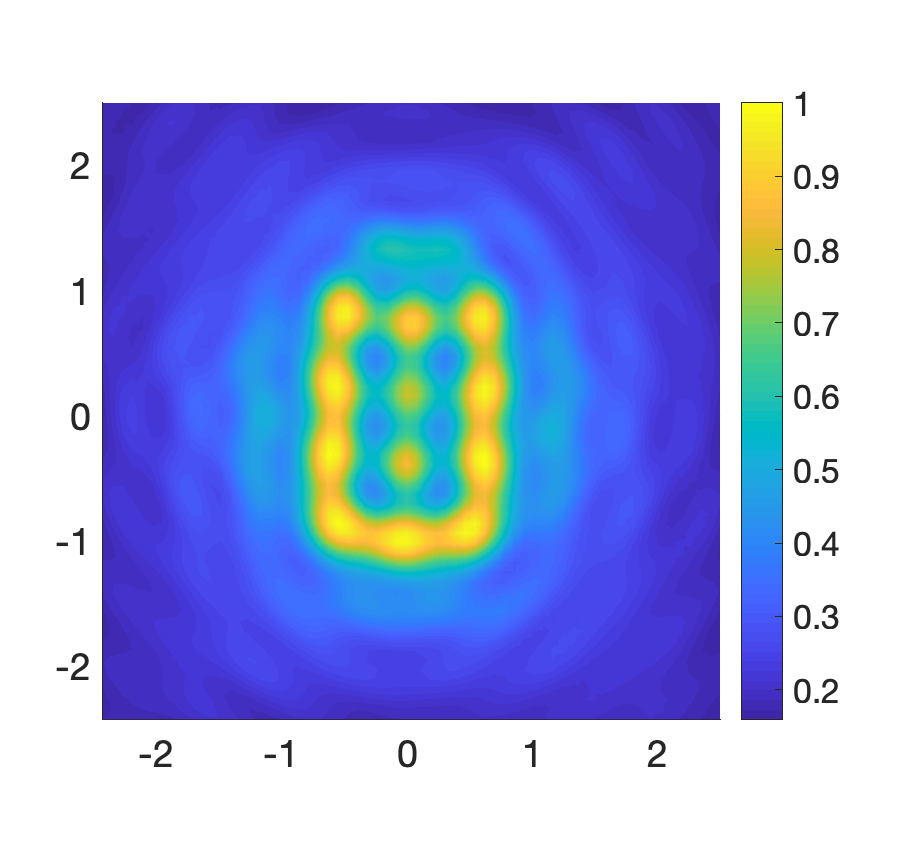}}  \hspace{-0.3cm} 
%\subfloat[]{\includegraphics[width=4.2cm]{rechole_k8_d96_r100_usc_I2}}  
\caption{Reconstruction with experimental  data from the Fresnel Institute using $I_{\text{far}}(z)$.
 } 
 \label{fi5}
\end{figure}

%
%\section{Summary}
%\label{summary}
%We propose two novel imaging functionals of orthogonality sampling type for solving 
%the inverse acoustic scattering problem with Cauchy data in $\R^n$ ($n = 2$ or 3). 
%These functionals  are fast, simple to implement, stable, and applicable to both near-field and far-field 
%data.  Both theoretical analysis and numerical results are presented. 

\vspace{0.5cm}
\textbf{Acknowledgement}. The work of the authors  was partially supported by NSF grant DMS-2208293.

\bibliographystyle{plain}
\bibliography{ip-biblio}

\end{document}